\documentclass[a4paper,reqno,11pt]{amsart}

\usepackage[utf8]{inputenc}
\usepackage{microtype}

\usepackage[margin=1in]{geometry}

\usepackage{amssymb}
\usepackage{mathrsfs}
\usepackage{mathtools}
\usepackage{enumitem}
\usepackage{longtable}

\usepackage{color}
\usepackage{tikz}
\usetikzlibrary{topaths}

\usepackage{hyperref}

\newtheorem{theorem}{Theorem}[section]
\newtheorem*{theorem*}{Theorem}

\newtheorem{lemma}[theorem]{Lemma}

\newtheorem{proposition}[theorem]{Proposition}
\newtheorem*{proposition*}{Proposition}

\newtheorem{corollary}[theorem]{Corollary}
\newtheorem*{corollary*}{Corollary}
\newtheorem{conjecture}[theorem]{Conjecture}
\newtheorem*{conjecture*}{Conjecture}

\newtheorem*{question*}{Question}
\newtheorem{cit}[theorem]{Citation}
\newtheorem{observation}[theorem]{Observation}

\newtheorem*{main:LM_htpy_infty_zero}{Theorem~\ref{thrm:LM_htpy_infty_zero}}
\newtheorem*{main:BNS}{Theorem~\ref{thrm:BNS}}

\theoremstyle{definition}
\newtheorem{definition}[theorem]{Definition}
\newtheorem{remark}[theorem]{Remark}

\newcommand{\Z}{\mathbb{Z}}
\newcommand{\N}{\mathbb{N}}

\newcommand{\R}{\mathbb{R}}

\newcommand{\LMsmall}{G}
\newcommand{\LMleft}{{}_yG}
\newcommand{\LMright}{G_y}
\newcommand{\LMbig}{{}_yG_y}

\newcommand{\onto}{\twoheadrightarrow}
\newcommand{\defeq}{\mathbin{\vcentcolon =}}

\DeclareMathOperator{\F}{F}

\numberwithin{equation}{section}

\begin{document}

\title[On the Lodha--Moore groups]{HNN decompositions of the Lodha--Moore 
groups, and topological applications}
\date{\today}


\author{Matthew C.~B.~Zaremsky}
\address{Department of Mathematical Sciences, Binghamton University, 
Binghamton, 
NY 13902}
\email{zaremsky@math.binghamton.edu}

\begin{abstract}
 The Lodha--Moore groups provide the first known examples of type $\F_\infty$ groups that are non-amenable and contain no non-abelian free subgroups. These groups are related to Thompson's group $F$ in certain ways, for instance they contain it as a subgroup in a natural way. We exhibit decompositions of four Lodha--Moore groups, $\LMsmall$, $\LMright$, $\LMleft$ and $\LMbig$, into ascending HNN extensions of isomorphic copies of each other, both in ways reminiscent to such decompositions for $F$ and also in quite different ways. This allows us to prove two new topological results about the Lodha--Moore groups. First, we prove that they all have trivial homotopy groups at infinity; in particular they are the first examples of groups satisfying all four parts of Geoghegan's 1979 conjecture about $F$. Second, we compute the Bieri--Neumann--Strebel invariant $\Sigma^1$ for the Lodha--Moore groups, and get some partial results for the Bieri--Neumann--Strebel--Renz invariants $\Sigma^m$, including a full computation of $\Sigma^2$.
\end{abstract}

\maketitle
\thispagestyle{empty}

\section*{Introduction}

The Lodha--Moore groups constructed by Lodha and Moore in \cite{lodha13,lodha14} provide the first known examples of type $\F_\infty$ groups that are non-amenable and contain no non-abelian free subgroups. These groups are closely related to Thompson's group $F$, and all contain it in a natural way. One Lodha--Moore group, which we denote by $\LMsmall$ and which was denoted $G_0$ in \cite{lodha13}, admits a presentation with only three generators and nine relations. The examples arise as subgroups of Monod's groups of piecewise projective homeomorphisms of the circle \cite{monod13}.

In 1979 Geoghegan made four conjectures about $F$, namely:
\begin{itemize}
 \item[(1)] It is of type $\F_\infty$.
 \item[(2)] It has no non-abelian free subgroups.
 \item[(3)] It is non-amenable.
 \item[(4)] It has trivial homotopy groups at infinity.
\end{itemize}
Brown and Geoghegan proved (1) and (4) \cite{brown84}, and Brin and Squier proved (2) \cite{brin85}. Conjecture (3) remains famously open. One can consider these four conjectures for the Lodha--Moore groups as well, in which case (2) and (3) were proved in \cite{lodha13} and (1) in \cite{lodha14}. This leaves only (4). In this paper we prove that the Lodha--Moore groups do indeed satisfy conjecture (4), and so provide the first examples of groups satisfying all four parts of Geoghegan's conjecture for $F$.

\begin{main:LM_htpy_infty_zero}
 The homotopy groups at infinity  of any Lodha--Moore group are trivial.
\end{main:LM_htpy_infty_zero}

The key tool is to exhibit decompositions of the Lodha--Moore groups into ascending HNN extensions of isomorphic copies of each other. In addition to $\LMsmall$, we consider Lodha--Moore-style groups $\LMright$, $\LMleft$ and $\LMbig$. The last of these was considered in \cite{lodha13}, denoted $G$ there, and the other two are obvious additions to the family. The groups are arranged via $F\subset \LMsmall \subset \LMleft,\LMright \subset \LMbig$. We prove:

\begin{theorem*}
 Both $\LMsmall$ and $\LMbig$ decompose as ascending HNN extensions of $\LMright$, and also of $\LMleft$. Both $\LMright$ and $\LMleft$ decompose as ascending HNN extensions of $\LMsmall$, and also of $\LMbig$.
\end{theorem*}

More precise formulations are found in Lemmas~\ref{lem:F_like_hnn} and~\ref{lem:weird_hnn}, and Corollaries~\ref{cor:F_like_hnns} and~\ref{cor:weird_hnns}.

Note that $F$ decomposes as an ascending HNN extension of an isomorphic copy of \emph{itself}. This was one key observation toward proving conjecture (4) for $F$ in \cite{brown84}. In particular, we consider this fact to be an example of the Lodha--Moore groups being ``$F$-like.'' In this same vein, the fact that these groups contain no non-abelian free groups follows in an essentially identical way to the same fact for $F$, from \cite{brin85}. In other ways though, the groups are not very $F$-like. For instance the non-amenability of the groups follows via arguments based on work of Ghys and Carri\`ere \cite{ghys85}, which have no chance of being adapted to $F$.

Obtaining HNN decompositions helps to prove the fourth part of the Geoghegan conjecture for these groups (Theorem~\ref{thrm:LM_htpy_infty_zero}), and has another interesting application as well. Namely, it helps us to compute their Bieri--Neumann--Strebel (BNS) invariants. The BNS-invariant of a group is a geometric object that reveals more subtle finiteness properties of the group; in particular it encodes information describing precisely which normal subgroups corresponding to abelian quotients are finitely generated. We prove:

\begin{main:BNS}
 The BNS-invariant $\Sigma^1(H)$ of any Lodha--Moore group $H$ is of the form $S^2 \setminus P$ for some set $P$ with $|P|=2$.
\end{main:BNS}

The actual formulation of this theorem, e.g., describing $P$, involves a lot of notation, which will be introduced in the sections leading up to the precise formulation, and for the sake of this introduction we will not give these details yet. Our computation of $\Sigma^1$ again falls into the category of ways in which the Lodha--Moore groups are $F$-like, in that their BNS-invariants are obtained from their character spheres by removing exactly two points.

After computing the BNS-invariants, we consider the higher BNSR-invariants $\Sigma^m$. We give strong evidence that they are all obtained just by removing the convex hull of the two points missing from $\Sigma^1$, as is also the case for $F$ \cite{bieri87,bieri10}. We prove this for the case $m=2$; see Theorem~\ref{thrm:BNSR} and Section~\ref{sec:Sigma^2}.

\medskip

The paper is organized as follows. In Section~\ref{sec:groups} we recall the construction of the groups, compute their abelianizations, and exhibit some important discrete characters, i.e., homomorphisms to $\Z$, for the groups. In Section~\ref{sec:HNN} we inspect how these groups decompose as strictly ascending HNN extensions of each other. In Section~\ref{sec:htpy_infty} we show that they have trivial homotopy groups at infinity. In Section~\ref{sec:inv} we recall the BNS-invariant, and some important tools to compute it, and then compute the BNS-invariants of all the Lodha--Moore groups. In Section~\ref{sec:bnsr} we recall the higher BNSR-invariants $\Sigma^m$ for $m\in\N\cup\{\infty\}$, and reduce the problem of computing them for the Lodha--Moore groups to proving a single conjecture, about one certain subgroup being of type $\F_\infty$ (Conjecture~\ref{conj:ker_psi_F_m}). We prove that this subgroup is at least finitely presented, and so obtain a complete computation of $\Sigma^2$ for each Lodha--Moore group (Theorem~\ref{thrm:BNSR} for $m=2$).

As a remark, it is recommended, though not strictly necessary, that the reader be familiar with Lodha and Moore's paper \cite{lodha13} before attempting to read the present work. On the other hand, specialized knowledge of homotopy at infinity or BNS-invariants is not particularly necessary before reading the present work.

\subsection*{Acknowledgments} I am grateful to Matt Brin, Ross Geoghegan and Yash Lodha for many helpful discussions. Also thanks to Ross Geoghegan and Mike Mihalik for both, independently, asking whether Theorem~\ref{thrm:LM_htpy_infty_zero} was true (and hence motivating what became Section~\ref{sec:htpy_infty}).


\section{The groups}\label{sec:groups}

We will first define $\LMbig$, and then define $\LMsmall,\LMleft,\LMright$ as subgroups of $\LMbig$.

We will view elements of the group $\LMbig$ as functions on the Cantor set $2^\N$. Thus an element will be specified by declaring how it acts on an arbitrary infinite string of $0$s and $1$s. The starting point is two primitive functions, $x$ and $y$, defined as follows. (Our actions are on the right, following \cite{lodha13}.)

\begin{align*}
 \xi.x \defeq \left\{\begin{array}{ll} 0\eta & \text{ if } \xi=00\eta\\
                                      10\eta & \text{ if } \xi=01\eta\\
                                      11\eta & \text{ if } \xi=1\eta\end{array}\right.
\end{align*}

\begin{align*}
 \xi.y \defeq \left\{\begin{array}{ll} 0(\eta.y) & \text{ if } \xi=00\eta\\
                                      10(\eta.y^{-1}) & \text{ if } \xi=01\eta\\
                                      11(\eta.y) & \text{ if } \xi=1\eta\end{array}\right.
\end{align*}

One can extrapolate how $x^{-1}$ and $y^{-1}$ act on $2^\N$ as well.

Now let $2^{<\N}$ be the set of finite binary sequences. For any $s\in 2^{<\N}$, we can define functions $x_s$ and $y_s$ that act like $x$ or $y$ ``at'' the address $s$ and otherwise act like the identity:

\begin{align*}
 \xi.x_s \defeq \left\{\begin{array}{ll} s(\eta.x) & \text{ if } \xi=s\eta\\
                                               \xi & \text{ otherwise}\end{array}\right.
\end{align*}

\begin{align*}
 \xi.y_s \defeq \left\{\begin{array}{ll} s(\eta.y) & \text{ if } \xi=s\eta\\
                                               \xi & \text{ otherwise}\end{array}\right.
\end{align*}

The group $\langle x_s\mid s\in 2^{<\N}\rangle$ is Thompson's group $F$.

We now define the Lodha--Moore groups as follows. Here $0^n$ and $1^n$ denote words consisting of a string of $n$ such symbols. In particular $0^0=1^0=\emptyset$, the empty word. Also, $\N_0$ denotes $\N\cup\{\infty\}$.

$$\LMbig\defeq\langle x_s, y_t\mid s,t\in 2^{<\N}\rangle \text{.}$$

$$\LMright\defeq\langle x_s, y_t\mid s,t\in 2^{<\N}\text{, } t\not\in\{0^n\}_{n\in\N_0}\rangle$$

$$\LMleft\defeq\langle x_s, y_t\mid s,t\in 2^{<\N}\text{, } t\not\in\{1^n\}_{n\in\N_0}\rangle$$

$$\LMsmall\defeq\langle x_s, y_t\mid s,t\in 2^{<\N}\text{, } t\not\in\{0^n,1^n\}_{n\in\N_0}\rangle$$

In words, the differences between the groups involve at which addresses in $2^{<\N}$ we allow a $y$ to act. In $\LMbig$, a $y$ can act anywhere, in $\LMright$, a $y$ can act anywhere except not at addresses of the form $0^n$, in $\LMleft$, a $y$ can act anywhere except not at addresses $1^n$, and in $\LMsmall$, a $y$ can act anywhere except not at addresses $0^n$ or $1^n$. The notation is a visual reminder of where $y$'s can act. For instance $\LMleft$ contains the element $y_0$ but not $y_1$.

In \cite{lodha13}, presentations are given for $\LMsmall$ and $\LMbig$. It is straightforward to extrapolate these to find presentations for $\LMleft$ and $\LMright$. We state the defining relations for $\LMbig$, and to restrict to the others we just restrict which subscripts are used for the $y$ generators. The relations are:

\begin{itemize}
 \item[(LM1)] $x_s^2=x_{s0} x_s x_{s1}$ for $s\in 2^{<\N}$
 \item[(LM2)] $x_t x_s = x_s x_{t.x_s}$ for $s,t\in 2^{<\N}$ with $t.x_s$ defined
 \item[(LM3)] $y_t x_s = x_s y_{t.x_s}$ for $s,t\in 2^{<\N}$ with $t.x_s$ defined
 \item[(LM4)] $y_s y_t = y_t y_s$ for $s,t\in 2^{<\N}$ with neither $s$ nor $t$ a prefix of the other
 \item[(LM5)] $y_s = x_s y_{s0} y_{s10}^{-1} y_{s11}$ for $s\in 2^{<\N}$
\end{itemize}

The action of $F$ on $2^\N$ restricts to a partial action of $x$ on $2^{<\N}$. When we say ``$t.x_s$ is defined, we mean that $x_s$ can act on $t$.

\subsection{Strand diagrams}\label{sec:strands}

In \cite{belk14}, Belk and Matucci develop the \emph{strand diagram} model for elements of $F$. This is related to the well known \emph{paired tree diagram} model. We are using right actions, and so here to get from a paired tree diagram to a strand diagram, pictorially we reflects the first (domain) tree upside down and attach its leaves to those of the second (range) tree. In \cite{lodha13}, Lodha and Moore discuss a paired tree diagram model for elements of their groups. It is natural then to also develop a strand diagram model, which is what we do in this subsection. We will not be overly rigorous, since this is just a helpful tool that makes computations easier.

First we discuss elements of $F$ represented as strand diagrams (see \cite{belk14} for more details). A strand diagram is a picture of a single strand splitting into multiple strands, always via bifurcation, and then the multiple strands merging two at a time, with possible further splitting and merging until ultimately everything merges back down into one strand. With our present convention, the merges represent the carets of the (now upside down) first tree, and the splits represent the carets of the second tree.

Two strand diagrams may represent the same element; the equivalence relation on such diagrams is generated by reduction moves of two forms, namely, those in the following picture:

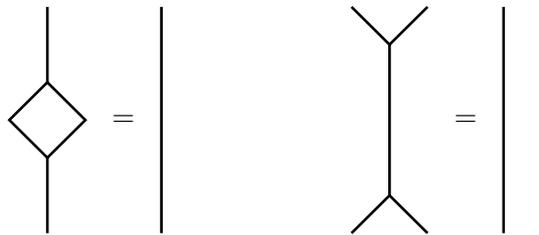
\begin{figure}[htb]
\begin{tikzpicture}[line width=1pt]
 \draw (0,0) -- (0,-1) -- (-0.5,-1.5) -- (0,-2) -- (0.5,-1.5) -- (0,-1)   (0,-2) -- (0,-3);
 \node at (1,-1.5) {$=$};
 \draw (1.5,0) -- (1.5,-3);
 \draw (4,0) -- (4.5,-0.5) -- (5,0)   (4.5,-0.5) -- (4.5,-2.5) -- (4,-3)   (4.5,-2.5) -- (5,-3);
 \node at (5.5,-1.5) {$=$};
 \draw (6,0) -- (6,-3)   (6.5,0) -- (6.5,-3);
\end{tikzpicture}
\caption{Fundamental reduction moves in strand diagrams for $F$.}\label{fig:F_strand_moves}
\end{figure}

Multiplication in $F$ is modeled by stacking strand diagrams and performing reduction moves. A product $ab$ corresponds to stacking the strand diagrams for $a$ and $b$. Since we are using right actions, in the product $ab$ we stack $b$ on top of $a$.

Now we introduce strand diagrams for elements of Lodha--Moore groups. Since $F=\langle x_s\rangle_s$, we already have pictures for the $x_s$ and we just need to invent pictures for the $y_s$. Moreover, the nodes of a finite binary tree are labeled by finite binary sequences, so once we have a picture for $y_\emptyset$, we will have one for any $y_s$, by putting the picture at the address $s$. So, we just need to declare how to represent $y_\emptyset$, and then check that the relations are correctly modeled. We choose as a strand diagram for $y_\emptyset$ a picture of a single strand with a counterclockwise ``cyclone'' $\circlearrowleft$ in the middle. Similarly $y_\emptyset^{-1}$ is a single strand with a clockwise cyclone $\circlearrowright$. Now elements of $\LMbig$ are represented by strand diagrams with cyclones on some strands. See Figure~\ref{fig:LM_strand_example} for an example.

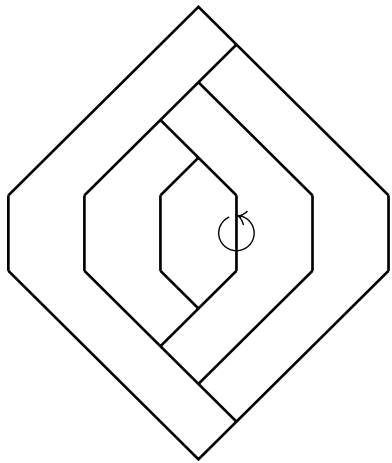
\begin{figure}[htb]
\begin{tikzpicture}[line width=1pt]
 \draw (0,0) -- (2.5,2.5) -- (5,0)   (1,0) -- (3,2)   (2.5,1.5) -- (4,0)   (2,1) -- (3,0)   (2,0) -- (2.5,0.5);
 \draw (0,0) -- (0,-1)   (1,0) -- (1,-1)   (2,0) -- (2,-1)   (3,0) -- (3,-1)   (4,0) -- (4,-1)   (5,0) -- (5,-1);
 \node at (3,-0.5) {\huge $\circlearrowleft$};
 \begin{scope}[yscale=-1,yshift=1cm]
  \draw (0,0) -- (2.5,2.5) -- (5,0)   (1,0) -- (3,2)   (2.5,1.5) -- (4,0)   (2,1) -- (3,0)   (2,0) -- (2.5,0.5);
 \end{scope}
\end{tikzpicture}
\caption{Strand diagram for $y_{10011}$.}\label{fig:LM_strand_example}
\end{figure}

It will not come up for us, but for an element like $y_s^2$, one could just draw two counterclockwise cyclones on the same strand.

These strand diagrams are still considered up to equivalence. The old reduction moves from $F$ still hold, and the new reduction moves include: if a counterclockwise cyclone and a clockwise cyclone share the same segment of a strand, we may delete both (corresponding to $y_s y_s^{-1}=y_s^{-1}y_s=1$), and the following ``expansion'' moves shown in Figure~\ref{fig:LM_strand_moves}:

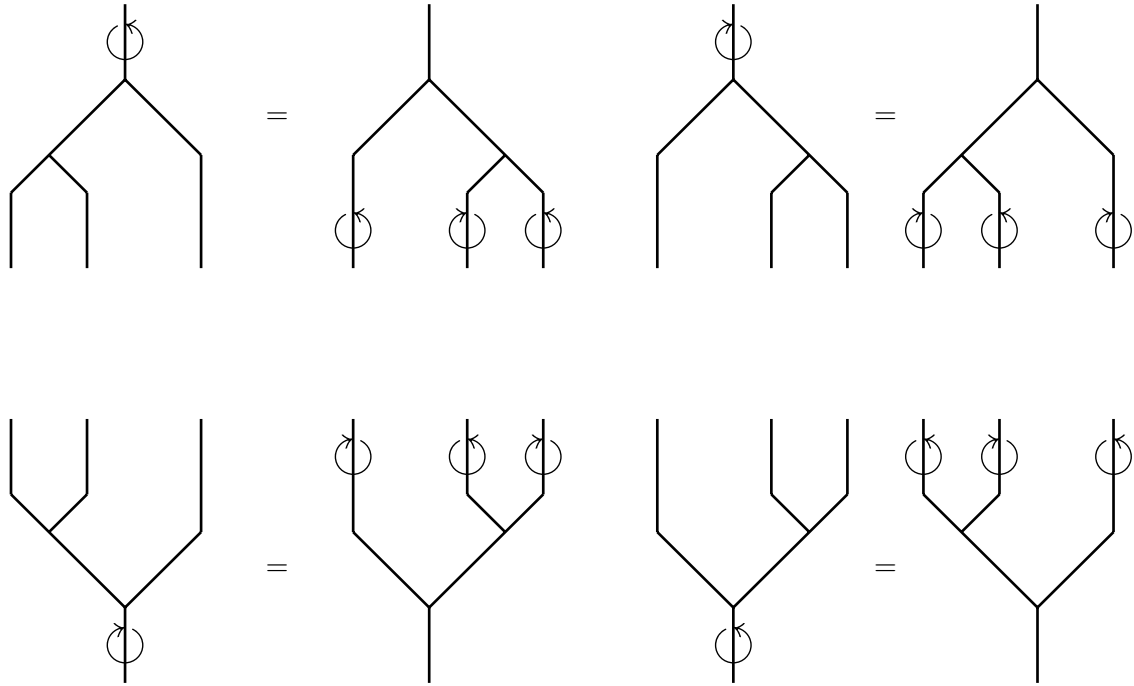
\begin{figure}[htb]
\begin{tikzpicture}[line width=1pt]
 \begin{scope}[xshift=-4cm,xscale=-1]
  \draw (0,1) -- (0,0)   (-1,-1) -- (0,0) -- (1.5,-1.5)   (0.5,-1.5) -- (1,-1)   (-1,-1) -- (-1,-2.5)   (1.5,-1.5) -- (1.5,-2.5)   (0.5,-1.5) -- (0.5,-2.5);
	\node at (0,0.5) {\huge $\circlearrowleft$}; \node at (-2,-0.5) {$=$};
 \end{scope}
 \draw (0,1) -- (0,0)   (-1,-1) -- (0,0) -- (1.5,-1.5)   (0.5,-1.5) -- (1,-1)   (-1,-1) -- (-1,-2.5)   (1.5,-1.5) -- (1.5,-2.5)   (0.5,-1.5) -- (0.5,-2.5);
 \node at (-1,-2) {\huge $\circlearrowleft$}; \node at (0.5,-2) {\huge $\circlearrowright$}; \node at (1.5,-2) {\huge $\circlearrowleft$};

\begin{scope}[xshift=4cm]
\draw (0,1) -- (0,0)   (-1,-1) -- (0,0) -- (1.5,-1.5)   (0.5,-1.5) -- (1,-1)   (-1,-1) -- (-1,-2.5)   (1.5,-1.5) -- (1.5,-2.5)   (0.5,-1.5) -- (0.5,-2.5);
 \node at (0,0.5) {\huge $\circlearrowright$}; \node at (2,-0.5) {$=$};
 \begin{scope}[xshift=4cm,xscale=-1]
  \draw (0,1) -- (0,0)   (-1,-1) -- (0,0) -- (1.5,-1.5)   (0.5,-1.5) -- (1,-1)  (-1,-1) -- (-1,-2.5)   (1.5,-1.5) -- (1.5,-2.5)   (0.5,-1.5) -- (0.5,-2.5);
	\node at (-1,-2) {\huge $\circlearrowright$}; \node at (1.5,-2) {\huge $\circlearrowright$}; \node at (0.5,-2) {\huge $\circlearrowleft$};
 \end{scope}
\end{scope}

\begin{scope}[yshift=-7cm,yscale=-1]
 \begin{scope}[xshift=-4cm,xscale=-1]
   \draw (0,1) -- (0,0)   (-1,-1) -- (0,0) -- (1.5,-1.5)   (0.5,-1.5) -- (1,-1)   (-1,-1) -- (-1,-2.5)   (1.5,-1.5) -- (1.5,-2.5)   (0.5,-1.5) -- (0.5,-2.5);
	 \node at (0,0.5) {\huge $\circlearrowright$}; \node at (-2,-0.5) {$=$};
  \end{scope}
  \draw (0,1) -- (0,0)   (-1,-1) -- (0,0) -- (1.5,-1.5)   (0.5,-1.5) -- (1,-1)   (-1,-1) -- (-1,-2.5)   (1.5,-1.5) -- (1.5,-2.5)   (0.5,-1.5) -- (0.5,-2.5);
  \node at (-1,-2) {\huge $\circlearrowright$}; \node at (0.5,-2) {\huge $\circlearrowleft$}; \node at (1.5,-2) {\huge $\circlearrowright$};

 \begin{scope}[xshift=4cm]
  \draw (0,1) -- (0,0)   (-1,-1) -- (0,0) -- (1.5,-1.5)   (0.5,-1.5) -- (1,-1)   (-1,-1) -- (-1,-2.5)   (1.5,-1.5) -- (1.5,-2.5)   (0.5,-1.5) -- (0.5,-2.5);
 \node at (0,0.5) {\huge $\circlearrowleft$}; \node at (2,-0.5) {$=$};
 \begin{scope}[xshift=4cm,xscale=-1]
  \draw (0,1) -- (0,0)   (-1,-1) -- (0,0) -- (1.5,-1.5)   (0.5,-1.5) -- (1,-1)   (-1,-1) -- (-1,-2.5)   (1.5,-1.5) -- (1.5,-2.5)   (0.5,-1.5) -- (0.5,-2.5);
	\node at (-1,-2) {\huge $\circlearrowleft$}; \node at (1.5,-2) {\huge $\circlearrowleft$}; \node at (0.5,-2) {\huge $\circlearrowright$};
 \end{scope}
\end{scope}
\end{scope}
\end{tikzpicture}
\caption{Four expansion moves in strand diagrams for Lodha--Moore groups.}\label{fig:LM_strand_moves}
\end{figure}

The point of using these cyclones is that the direction of spinning tells us whether to ``push'' $\{0,10,11\}$ to $\{00,01,1\}$ or vice versa. Our convention of reading the diagrams bottom-to-top ensures that positive powers of $y_s$ generators correspond to counterclockwise (positive) cyclones.

To demonstrate that our conventions for strand diagrams correctly model the groups, we draw some examples of the defining relations in Figures~\ref{fig:pent},~\ref{fig:y_conj} and~\ref{fig:expand}.

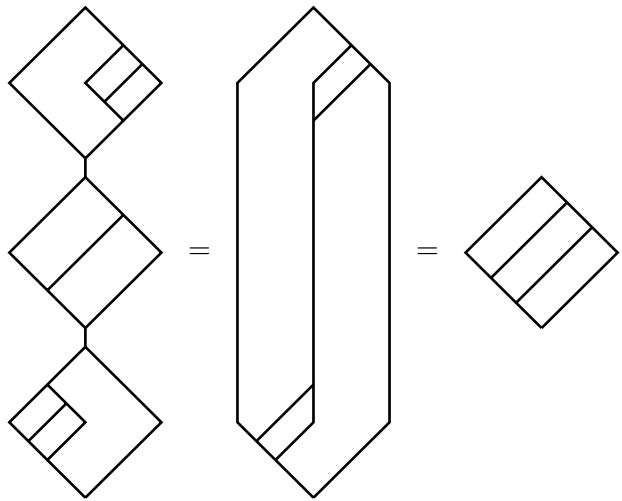
\begin{figure}[htb]
\begin{tikzpicture}[line width=1pt,yscale=-1]
 \draw (0,0) -- (-1,-1) -- (0,-2) -- (1,-1) -- (0,0)   (-0.5,-0.5) -- (0,-1) -- (-0.5,-1.5)   (-0.75,-0.75) -- (-0.25,-1.25);
 \draw (0,-2) -- (0,-2.25);
 \draw (0,-2.25) -- (-1,-3.25) -- (0,-4.25) -- (1,-3.25) -- (0,-2.25)   (-0.5,-2.75) -- (0.5,-3.75);
 \draw (0,-4.25) -- (0,-4.5);
 \draw (0,-4.5) -- (-1,-5.5) -- (0,-6.5) -- (1,-5.5) -- (0,-4.5)   (0.5,-5) -- (0,-5.5) -- (0.5,-6)   (0.25,-5.25) -- (0.75,-5.75);
 \node at (1.5,-3.25) {$=$};
 \draw (3,0) -- (2,-1) -- (2,-5.5) -- (3,-6.5) -- (4,-5.5) -- (4,-1) -- (3,0)   (2.5,-0.5) -- (3,-1) -- (3,-5.5) -- (3.5,-6);
 \draw (2.25,-0.75) -- (3,-1.5)   (3.75,-5.75) -- (3,-5);
 \node at (4.5,-3.25) {$=$};
 \draw (6,-2.25) -- (5,-3.25) -- (6,-4.25) -- (7,-3.25) -- (6,-2.25)   (5.33,-2.91) -- (6.33,-3.91)   (5.66,-2.58) -- (6.66,-3.58);
\end{tikzpicture}
\caption{The relation $x_0 x_\emptyset x_1 = x_\emptyset^2$.}\label{fig:pent}
\end{figure}

\begin{figure}[htb]
\begin{tikzpicture}[line width=1pt]
 \draw (0,2.25) -- (-1,1.25) -- (0,0.25) -- (1,1.25) -- (0,2.25)   (0.5,1.75) -- (-0.5,0.75);
 \draw (0,0) -- (0,0.25);
 \draw (0,0) -- (-1,-1) -- (0,-2) -- (1,-1) -- (0,0)   (0.5,-0.5) -- (0,-1) -- (0.5,-1.5);
 \node at (1,-1) {\huge $\circlearrowleft$};
 \draw (0,-2) -- (0,-2.25);
 \draw (0,-2.25) -- (-1,-3.25) -- (0,-4.25) -- (1,-3.25) -- (0,-2.25)   (-0.5,-2.75) -- (0.5,-3.75);
 \node at (1.5,-2.125) {$=$};
 \draw (3,-1.125) -- (2,-2.125) -- (3,-3.125) -- (4,-2.125) -- (3,-1.125)   (3.5,-1.625) -- (3,-2.125) -- (3.5,-2.625);
 \node at (4,-2.125) {\huge $\circlearrowleft$};
 \draw (3+0.33,-1.125-0.33) -- (3-0.33,-2.125) -- (3+0.33,-3.125+0.33);
\end{tikzpicture}
\caption{The relation $x_\emptyset^{-1} y_1 x_\emptyset=y_{11}$.}\label{fig:y_conj}
\end{figure}

\begin{figure}[htb]
\begin{tikzpicture}[line width=1pt,yscale=-1]
 \draw (0,0) -- (-1,-1) -- (0,-2) -- (1,-1) -- (0,0)   (0.5,-0.5) -- (0,-1) -- (-0.5,-1.5);
 \draw (0,-2) -- (0,-3);
 \node at (0,-2.5) {\huge $\circlearrowleft$};
 \node at (1.5,-1.5) {$=$};
 \draw (2,-1) -- (3,0) -- (4,-1)   (3.5,-0.5) -- (3,-1);
 \draw (2,-1) -- (2,-2)   (3,-1) -- (3,-2)   (4,-1) -- (4,-2);
 \draw (2,-2) -- (3,-3) -- (4,-2)   (3.5,-2.5) -- (3,-2);
 \node at (2,-1.5) {\huge $\circlearrowleft$}; \node at (3,-1.5) {\huge $\circlearrowright$}; \node at (4,-1.5) {\huge $\circlearrowleft$};
\end{tikzpicture}
\caption{The relation $x_\emptyset^{-1} y_\emptyset = y_0 y_{10}^{-1} y_{11}$.}\label{fig:expand}
\end{figure}
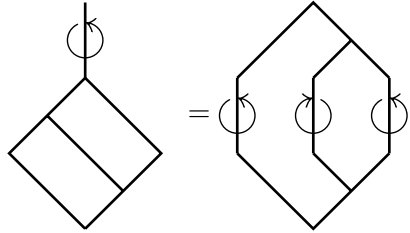

\subsection{Abelianizations and characters}\label{sec:chars}

In this subsection we will abelianize the four Lodha--Moore groups; it turns out all of them abelianize to $\Z^3$. We will use bars to indicate the abelianized elements, and will write $H^{ab}$ for the abelianization of a group $H$.

\begin{lemma}\label{lem:abln_gens}
 $\LMsmall^{ab}$ is generated by $\overline{x}_0$, $\overline{x}_1$ and $\overline{y}_{10}$. $\LMleft^{ab}$ is generated by $\overline{y}_0$, $\overline{x}_1$ and $\overline{y}_{10}$. $\LMright^{ab}$ is generated by $\overline{x}_0$, $\overline{y}_1$ and $\overline{y}_{10}$. $\LMbig^{ab}$ is generated by $\overline{y}_0$, $\overline{y}_1$ and $\overline{y}_{10}$.
\end{lemma}

\begin{proof}
 We have \emph{a priori} that $\LMsmall^{ab}$ is generated by the $\overline{x}_s$ for all $s$ and the $\overline{y}_s$ for all $s\not\in\{0^n,1^n\}_{n\in\N_0}$. Note that the partial action of $F$ on $2^{<\N}$ has four orbits, represented by $\emptyset$, $0$, $1$ and $10$. By relation (LM3), every $\overline{y}_s$ equals $\overline{y}_{10}$. By relation (LM2), every $\overline{x}_s$ equals one of $\overline{x}_\emptyset$, $\overline{x}_0$, $\overline{x}_1$ or $\overline{x}_{10}$. Also, $\overline{x}_{10}=0$ and $\overline{x}_\emptyset=\overline{x}_0+\overline{x}_1$ since this is true in $F$.

 Now consider $\LMleft^{ab}$. Similar to before we reduce the generating set down to just $\overline{x}_0$, $\overline{x}_1$, $\overline{y}_0$ and $\overline{y}_{10}$. Then we use relation (LM5) to get
$$\overline{y}_0=\overline{x}_0 + \overline{y}_{00} - \overline{y}_{010} + \overline{y}_{011} = \overline{x}_0 + \overline{y}_{00} = \overline{x}_0 + \overline{y}_0$$
 whence $\overline{x}_0=0$. This finishes $\LMleft$. For $\LMright$ and $\LMbig$ we use parallel arguments.
\end{proof}

Now to see that these abelianizations are all $\Z^3$, we need to show that in each case the three generators are linearly independent. For each group it suffices to exhibit, for each of its three generators, an epimorphism from the group to $\Z$, i.e., a \emph{discrete character}, that takes that generator to $1$ and the other two generators to $0$. We define such characters now. After giving them names we will discuss which ones are well defined on which groups, and which generators they ``detect'' for the purpose to proving linear independence.

The characters we now define are denoted
$$\chi_0 \text{, } \chi_1 \text{, } \psi_0 \text{, } \psi_1 \text{ and } \psi \text{.}$$

First define $\chi_0$ by setting $\chi_0(x_{0^n})=-1$ for $n\ge0$, $\chi_0(x_s)=0$ for $s\neq 0^n$ and $\chi_0(y_s)=0$ for all $s$. Restricted to $F$, this is the character typically called $\chi_0$. Next define $\chi_1$ via $\chi_1(x_{1^n})=1$ for $n\ge0$, $\chi_1(x_s)=0$ for $s\neq 1^n$ and $\chi_1(y_s)=0$ for all $s$. Again, this is the usual character called $\chi_1$ when restricted to $F$.

Next define $\psi_0$ via $\psi_0(y_{0^n})=1$ for $n\ge0$, $\psi_0(y_s)=0$ for $s\neq 0^n$ and $\psi_0(x_s)=0$ for all $s$. Similarly define $\psi_1$ via $\psi_1(y_{1^n})=1$ for $n\ge0$, $\psi_1(y_s)=0$ for $s\neq 1^n$ and $\psi_1(x_s)=0$ for all $s$. Lastly define $\psi$ via $\psi(y_s)=1$ for all $s$ and $\psi(x_s)=0$ for all $s$.

For example, if $w=x_0^2 y_{000}^{-1} y_{11}^4$ then the five functions $\chi_0$, $\chi_1$, $\psi_0$, $\psi_1$ and $\psi$ read $-2$, $0$, $-1$, $4$ and $3$ respectively. (Though, as the next observation points out, we shouldn't try to apply $\chi_0$ or $\chi_1$ to this $w$ as an element of $\LMbig$, since it will not be well defined.)

\begin{observation}
 The characters $\chi_0,\chi_1,\psi$ are well defined on $\LMsmall$. The characters $\chi_0,\psi_1,\psi$ are well defined on $\LMright$. The characters $\psi_0,\chi_1,\psi$ are well defined on $\LMleft$. The characters $\psi_0,\psi_1,\psi$ are well defined on $\LMbig$.
\end{observation}

\begin{proof}
 We just need to verify that the relations (LM1) through (LM5) remain valid after applying any of these characters, and this is a straightforward exercise. For the reader interested in checking this, we reiterate that we only only need to check the relations with $y$-subscripts allowed in whichever Lodha--Moore group we are working with. For example, relation (LM5) for $s=\emptyset$ is $y_\emptyset = x_\emptyset y_0 y_{10}^{-1} y_{11}$, and applying $\chi_0$ we get $0=-1$; hence $\chi_0$ is not well defined on $\LMbig$. In $\LMsmall$ though, this relation does not appear, and in fact $\chi_0$ is well defined on $\LMsmall$.
\end{proof}

Characters provide a way to test linear independence of abelianized elements. For example, in $F$, if $a\overline{x}_0 + b\overline{x}_1 = 0$ then applying $\chi_0$ we see $-a+0=0$ and applying $\chi_1$ we see $0+b=0$, so in fact $a=b=0$ and we conclude that $\overline{x}_0$ and $\overline{x}_1$ are linearly independent. This shows that $F^{ab}\cong\Z^2$, and we have the following corollary for the Lodha--Moore groups:

\begin{corollary}
 The abelianizations of the Lodha--Moore groups are all isomorphic to $\Z^3$.
\end{corollary}

\begin{proof}
 We know from Lemma~\ref{lem:abln_gens} that $\LMsmall^{ab}$ is generated by $\overline{x}_0$, $\overline{x}_1$ and $\overline{y}_{10}$. We can tell these are linearly independent by applying the functions $\chi_0$, $\chi_1$ and $\psi$. For $\LMright^{ab}$ with generators $\overline{x}_0$, $\overline{y}_1$ and $\overline{y}_{10}$ we use $\chi_0$, $\psi_1$ and $\psi$. For $\LMleft^{ab}$ with generators $\overline{y}_0$, $\overline{x}_1$ and $\overline{y}_{10}$ we use $\psi_0$, $\chi_1$ and $\psi$. Finally, for $\LMbig^{ab}$ with generators $\overline{y}_0$, $\overline{y}_1$ and $\overline{y}_{10}$ we use $\psi_0$, $\psi_1$ and $\psi$.
\end{proof}

\section{HNN decompositions}\label{sec:HNN}

In this section we decompose the Lodha--Moore groups into ascending HNN extensions of each other. The eight HNN decompositions we find, (HNN1) through (HNN8), are listed at the end of the section for reference. First we give some general background, and then we will inspect the Lodha--Moore groups in two subsections.

\begin{definition}[(External) ascending HNN extension]\label{def:ext_hnn}
 Let $B$ be a group and $\phi\colon B\hookrightarrow B$ an injective endomorphism. The \emph{ascending HNN extension} of $B$ with respect to $\phi$ is the group
$$B*_{\phi,t}=\langle B,t\mid b^t=\phi(b) \text{ for all } b\in B\rangle \text{,}$$
where $b^t$ means $t^{-1}bt$. We call $t$ the \emph{stable element} of $B*_{\phi,t}$. If $\phi$ is not surjective, we call this a \emph{strictly ascending HNN extension}.
\end{definition}

Since we will be asking whether pre-existing groups have the form of ascending HNN extensions, we will take the following lemma as an alternate definition. (See also \cite[Lemma~3.1]{geoghegan01}.)

\begin{lemma}[(Internal) ascending HNN extension]\label{lem:int_hnn}
 Let $H$ be a group, $B\le H$ and $z\in H$. Suppose that $B$ and $z$ generate $H$, and that $B^z\subseteq B$. Suppose $z^n\in B$ only if $n=0$ (if $B^z\subsetneq B$ this is automatically satisfied). Let $\phi\colon B\hookrightarrow B$ be the map $b\mapsto b^z$. Then $H\cong B*_{\phi,t}$.
\end{lemma}

\begin{proof}
 Define an epimorphism $\Phi \colon B*_{\phi,t} \onto H$ by sending $t$ to $z$ and $b$ to $b$ for $b\in B$. This is a well defined homomorphism since $\Phi(b^t)=b^z=\phi(b)=\Phi(\phi(b))$ for all $b\in B$. We just need it to be injective. Elements of the abstract HNN extension $B*_{\phi,t}$ admit the form $t^n b t^m$ for $n\ge 0$, $b\in B$ and $m\le 0$. Suppose such an element lies in $\ker(\Phi)$. Then $z^n b z^m=1$, so $b=z^{-n-m}$. Hence $b=1$, and so indeed $\ker(\Phi)=\{1\}$.
\end{proof}

We will sometimes suppress the $\phi$ from the notation and just write $B*_t$, since when $B$ and $t$ live in an ambient group, $\phi$ is just conjugation by $t$. We call $H=B*_t$ an \emph{HNN decomposition} of $H$.

As we will see, the Lodha--Moore groups all decompose into interesting strictly ascending HNN extensions. Some of the decompositions are reminiscent of the case of Thompson's group $F$, but others are quite different.

\subsection{$F$-like HNN decompositions}\label{sec:Flike_HNN}

First we discuss the ``$F$-like'' decompositions. For $H\in\{\LMsmall,\LMleft,\LMright,\LMbig\}$ and $s\in 2^{<\N}$, define $H(s)$ to 
be the subgroup of $H$ generated by those $x_t$ and $y_t$ where $t$ extends $s$. Note that $\LMsmall(0)\cong \LMright$, via $x_{0t}\mapsto x_t$ and $y_{0t}\mapsto y_t$. (The important thing to note is that $y_{01^n}\mapsto y_{1^n}$, so $\LMright$ is the correct target, not $\LMsmall$.) More generally, such arguments reveal the following relationships:

\begin{itemize}
 \item $\LMsmall(0)\cong\LMright$
 \item $\LMsmall(1)\cong\LMleft$
 \item $\LMright(0)\cong\LMright$
 \item $\LMright(1)\cong\LMbig$
 \item $\LMleft(0)\cong\LMbig$
 \item $\LMleft(1)\cong\LMleft$
 \item $\LMbig(0)\cong\LMbig$
 \item $\LMbig(1)\cong\LMbig$
\end{itemize}

We will call $H(s)$ the \emph{semi-deferred} subgroup of $H$ at $s$. The prefix ``semi-'' is a reminder that $H(s)$ might not be isomorphic to $H$, in contrast to the case of $F$, when $F(s)\cong F$ is appropriately called the deferred copy of $F$ at $s$.

Recall that $F$ is a strictly ascending HNN extension of $F(1)\cong F$ with stable element $x_\emptyset$; it is also a strictly ascending HNN extension of $F(0)\cong F$ with stable element $x_\emptyset^{-1}$. We have some similar statements about the Lodha--Moore groups. We will prove the first such fact, and then list the other ones, which follow by parallel proofs, in a corollary.

\begin{lemma}\label{lem:F_like_hnn}
 We have that $\LMsmall$ is a strictly ascending HNN extension of $\LMsmall(1)$ with stable element $x_\emptyset$.
\end{lemma}

\begin{proof}
 First note that $\LMsmall$ is generated by $\LMsmall(1)$ and $x_\emptyset$. Indeed, the only generators missing are of the form $x_{0u}$ and $y_{0u}$; if $u$ features at least one $1$, then $x_{0u}$, respectively $y_{0u}$, is conjugate via $x_\emptyset$ to a generator of the form $x_{1t} \in \LMsmall(1)$, respectively $y_{1t} \in \LMsmall(1)$. Also, any $x_{0^n}$ for $n\ge1$ is conjugate via $x_\emptyset$ to $x_0 = x_\emptyset^2 x_1^{-1} x_\emptyset^{-1}$. So indeed we catch all the generators. Lastly note that $\LMsmall(1)^{x_\emptyset} \subseteq \LMsmall(11)$ since $x_{1t}^{x_\emptyset}=x_{11t}$ and $y_{1t}^{x_\emptyset}=y_{11t}$, so $\LMsmall(1)^{x_\emptyset} \subsetneq \LMsmall(1)$.
\end{proof}

\begin{corollary}\label{cor:F_like_hnns}
 By parallel proofs to that of the lemma, we have that $\LMsmall$ is a strictly ascending HNN extension of $\LMsmall(0)$ with stable element $x_\emptyset^{-1}$. Moreover, $\LMright$ (respectively $\LMleft$) is a strictly ascending HNN extension of $\LMright(1)$ (respectively $\LMleft(0)$) with stable element $x_\emptyset$ (respectively $x_\emptyset^{-1}$).
\end{corollary}

\subsection{Non-$F$-like HNN decompositions}\label{sec:nonFlike_HNN}

Now we discuss some ways in which the Lodha--Moore groups decompose into strictly ascending HNN extensions that are not ``$F$-like''. More precisely, this time the base subgroup will not be a semi-deferred copy of a larger Lodha--Moore group, but rather the natural embedded subgroup copy of a smaller Lodha--Moore group. The stable elements will be $y_0^{-1}$ and $y_1$ instead of $x_\emptyset^\pm$. In this subsection, we treat $\LMsmall$ as a subgroup of $\LMleft$ and $\LMright$, and those as subgroups of $\LMbig$, all in the natural way.

First, we have a technical lemma.

\begin{lemma}\label{lem:weird_conj}
 We have $y_0 x_\emptyset y_0^{-1} \in \LMsmall$ and $y_0^{-1} x_0 y_0 \not\in \LMright$.
\end{lemma}

\begin{proof}
 Using the relations (LM1) through (LM5), we see that $y_s=y_{s1} y_{s01}^{-1} y_{s00} x_s$ for all $s$.
 
 Now using (LM1) through (LM5) and this new relation, we calculate:
\begin{align*}
 y_0 x_\emptyset y_0^{-1} &= (y_{01} y_{001}^{-1} y_{000} x_0) x_\emptyset (y_{01} y_{001}^{-1} y_{000} x_0)^{-1} \\
&= y_{01} y_{001}^{-1} y_{000} (x_0 x_\emptyset x_0^{-1}) y_{000}^{-1} y_{001} y_{01}^{-1} \\
&= y_{01} y_{001}^{-1} y_{000} x_\emptyset^2 x_1^{-1} x_0^{-1} y_{000}^{-1} y_{001} y_{01}^{-1} \\
&= x_\emptyset^2 y_{110} y_{10}^{-1} x_1^{-1} (y_{0} x_0^{-1}) y_{000}^{-1} y_{001} y_{01}^{-1} \\
&= x_\emptyset^2 y_{110} y_{10}^{-1} x_1^{-1} y_{01} y_{001}^{-1} y_{000} y_{000}^{-1} y_{001} y_{01}^{-1} \\
&= x_\emptyset^2 y_{110} y_{10}^{-1} x_1^{-1} \in \LMsmall \text{.}
\end{align*}

 To see that $y_0^{-1} x_0 y_0 \not\in \LMright$, we compute that $y_0^{-1} x_0 y_0 = y_{011}^{-1} y_{010} y_{00}^{-1} y_0$, and this is in $\LMright$ if and only if $y_{00}^{-1}y_0$ is. But this word is in \emph{standard form} (\cite[Definition~5.1]{lodha13}), and uses generators of the form $y_{0^n}$, so cannot lie in $\LMright$, by arguments similar to those in \cite[Section~5]{lodha13}.
\end{proof}

Figure~\ref{fig:weird_conj} shows that $y_0 x_\emptyset y_0^{-1} \in \LMsmall$ using strand diagrams.

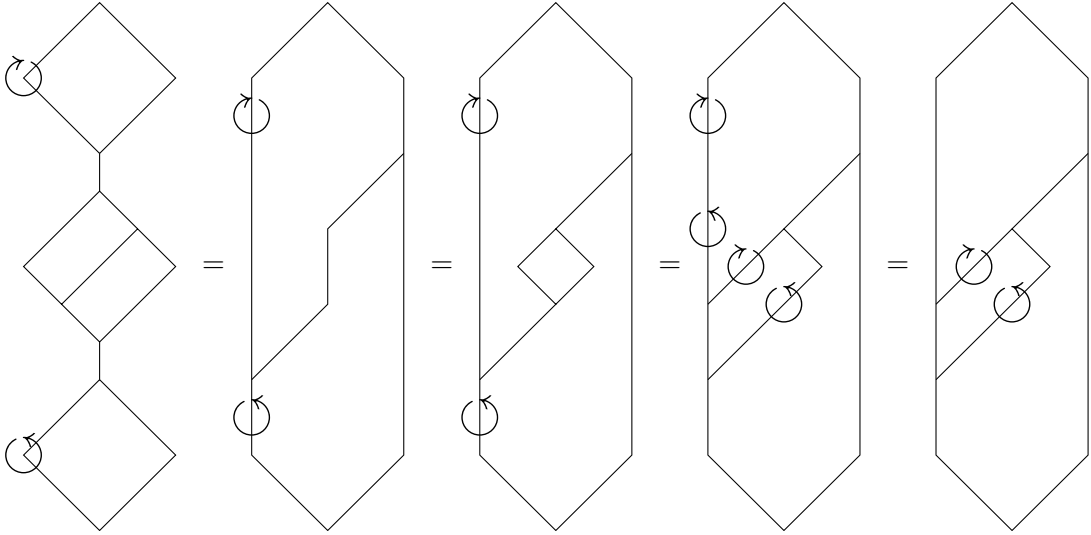
\begin{figure}[htb]
\begin{tikzpicture}[yscale=-1]
 \draw (0,0) -- (-1,-1) -- (0,-2) -- (1,-1) -- (0,0);
 \node at (-1,-1) {\huge $\circlearrowleft$};
 \draw (0,-2) -- (0,-2.5);
 \draw (0,0-2.5) -- (-1,-1-2.5) -- (0,-2-2.5) -- (1,-1-2.5) -- (0,0-2.5)   (-0.5,-0.5-2.5) -- (0.5,-1.5-2.5);
 \draw (0,-4.5) -- (0,-5);
 \draw (0,0-5) -- (-1,-1-5) -- (0,-2-5) -- (1,-1-5) -- (0,0-5);
 \node at (-1,-1-5) {\huge $\circlearrowright$}; \node at (1.5,-3.5) {$=$};
 \begin{scope}[xshift=3cm]
  \draw (0,0) -- (-1,-1) -- (-1,-6) -- (0,-7) -- (1,-6) -- (1,-1) -- (0,0);
  \draw (-1,-2) -- (0,-3) -- (0,-4) -- (1,-5);
	\node at (-1,-1.5) {\huge $\circlearrowleft$}; \node at (-1,-5.5) {\huge $\circlearrowright$}; \node at (1.5,-3.5) {$=$};
 \end{scope}
 \begin{scope}[xshift=6cm]
  \draw (0,0) -- (-1,-1) -- (-1,-6) -- (0,-7) -- (1,-6) -- (1,-1) -- (0,0);
  \draw (-1,-2) -- (0,-3) -- (-0.5,-3.5) -- (0,-4) -- (0.5,-3.5) -- (0,-3)   (0,-4) -- (1,-5);
	\node at (-1,-1.5) {\huge $\circlearrowleft$}; \node at (-1,-5.5) 
{\huge $\circlearrowright$}; \node at (1.5,-3.5) {$=$};
 \end{scope}
 \begin{scope}[xshift=9cm]
  \draw (0,0) -- (-1,-1) -- (-1,-6) -- (0,-7) -- (1,-6) -- (1,-1) -- (0,0);
  \draw (-1,-2) -- (0,-3)   (-1,-3) -- (-0.5,-3.5) -- (0,-4) -- (0.5,-3.5) -- (0,-3)   (0,-4) -- (1,-5);
	\node at (-1,-4) {\huge $\circlearrowleft$}; \node at (-1,-5.5) {\huge $\circlearrowright$}; \node at (-0.5,-3.5) {\huge $\circlearrowright$}; \node at (0,-3) {\huge $\circlearrowleft$}; \node at (1.5,-3.5) {$=$};
 \end{scope}
 \begin{scope}[xshift=12cm]
  \draw (0,0) -- (-1,-1) -- (-1,-6) -- (0,-7) -- (1,-6) -- (1,-1) -- (0,0);
  \draw (-1,-2) -- (0,-3)   (-1,-3) -- (-0.5,-3.5) -- (0,-4) -- (0.5,-3.5) -- (0,-3)   (0,-4) -- (1,-5);
	\node at (-0.5,-3.5) {\huge $\circlearrowright$}; \node at (0,-3) {\huge $\circlearrowleft$};
 \end{scope}
\end{tikzpicture}
\caption{A visual proof that $y_0 x_\emptyset y_0^{-1} \in \LMsmall$.}\label{fig:weird_conj}
\end{figure}

By similar proofs, we have that $y_1^{-1} x_\emptyset y_1 \in \LMsmall$ and $y_1 x_1 y_1^{-1} \not\in \LMleft$.

\begin{lemma}\label{lem:weird_hnn}
 We have that $\LMbig$ is a strictly ascending HNN extension of its subgroup $\LMright$ with stable element $y_0^{-1}$.
\end{lemma}

\begin{proof}
 First note that $\LMbig$ is generated by $\LMright=\langle x_\emptyset, x_1, y_1, y_{10}\rangle$ and $y_0$. By Lemma~\ref{lem:weird_conj}, we have $x_\emptyset^{y_0^{-1}}\in \LMright$, and we also know that $x_1^{y_0^{-1}}=x_1$, $y_1^{y_0^{-1}}=y_1$ and $y_{10}^{y_0^{-1}}=y_{10}$, so in fact $\LMright^{y_0^{-1}}\subseteq \LMright$. That this inclusion is proper follows from Lemma~\ref{lem:weird_conj} as well, since the lemma says $x_0\not\in \LMright^{y_0^{-1}}$.
\end{proof}

\begin{corollary}\label{cor:weird_hnns}
 By parallel proofs to the proofs of the lemmas, we also see that $\LMleft$ is a strictly ascending HNN extension of $\LMsmall$ with stable element $y_0^{-1}$, and that $\LMbig$ (respectively $\LMright$) is a strictly ascending HNN extension of $\LMleft$ (respectively $\LMsmall$) with stable element $y_1$.
\end{corollary}

\medskip

There are a lot of different HNN extensions to keep track of now, so we collect all the results of this section here, and name each HNN decomposition for reference. Here, whenever a smaller group is the base of a larger group, it is via the natural inclusion, e.g., $\LMsmall\subseteq \LMright$, and whenever a larger group is the base of a smaller group, it is via an isomorphism with a semi-deferred version, e.g., $\LMleft\cong \LMsmall(1) \subseteq \LMsmall$.

\begin{align*}\begin{array}{lcll}
 (\textrm{HNN1}) &\LMsmall&\cong&(\LMleft)*_{x_\emptyset} \\
 (\textrm{HNN2}) &\LMsmall&\cong&(\LMright)*_{x_\emptyset^{-1}} \\
 (\textrm{HNN3}) &\LMright&\cong&(\LMbig)*_{x_\emptyset} \\
 (\textrm{HNN4}) &\LMright&\cong&\LMsmall*_{y_1} \\
 (\textrm{HNN5}) &\LMleft&\cong&\LMsmall*_{y_0^{-1}} \\
 (\textrm{HNN6}) &\LMleft&\cong&(\LMbig)*_{x_\emptyset^{-1}} \\
 (\textrm{HNN7}) &\LMbig&\cong&(\LMright)*_{y_0^{-1}} \\
 (\textrm{HNN8}) &\LMbig&\cong&(\LMleft)*_{y_1} \end{array}
\end{align*}

The astute reader might notice some missing, like $\LMright\cong (\LMright)*_{x_\emptyset^{-1}}$ (via $\LMright(0)\cong \LMright$), but it turns out these eight are the only ones that will be usable for our purposes.

\begin{remark}
 There are some other strictly ascending HNN extensions related to these groups that we will not make use of, but which seem worth mentioning. First, one can calculate that $y_1^{-1}y_0 x_\emptyset y_0^{-1} y_1 = x_\emptyset^2$ (this was essentially done at the end of Section~4 of~\cite{lodha13}, and the reader is encouraged to verify it using strand diagrams), and so we can build a strictly ascending HNN extension $\LMsmall*_{y_0^{-1} y_1}$. This contains the group $\langle \sqrt{x_\emptyset},x_1\rangle$, where $\sqrt{x_\emptyset}\defeq y_1 y_0^{-1} x_\emptyset y_0 y_1^{-1}$ (so called since $(\sqrt{x_\emptyset})^2 = x_\emptyset$), which is a non-amenable, free group-free group with only two generators (it is not known whether it is finitely presented).
\end{remark}

\begin{remark}
 As the previous remark stated, $y_1^{-1}y_0 x_\emptyset y_0^{-1} y_1 = x_\emptyset^2$, and it is easy to check that $\LMsmall$ contains a copy of the Baumslag--Solitar group $BS(2,1)$. On the other hand, Thompson's group $V$ does not contain $BS(2,1)$ \cite{roever99}, so $\LMsmall$ does not embed into $V$. It is conjectured that every co$\mathcal{CF}$ group embeds into $V$ \cite{bleak13}, and so if $\LMsmall$ turned out to be co$\mathcal{CF}$, then it would be a counterexample.
\end{remark}


\section{Vanishing homotopy at infinity}\label{sec:htpy_infty}

Having decomposed the Lodha--Moore groups as ascending HNN extensions of each other, we can now apply results from \cite{mihalik85} and \cite{geoghegan08}, and techniques of Brown--Geoghegan \cite{brown84}, to quickly derive Theorem~\ref{thrm:LM_htpy_infty_zero}, that the groups have trivial homotopy groups at infinity. The reader interested in definitions and background is directed to Chapters 16 and 17 of \cite{geoghegan08}. We will at least quickly state a definition:

\begin{definition}\cite[Sections~17.1, 17.2]{geoghegan08}
 Let $H$ be a group of type $\F_n$ and let $X$ be a $K(H,1)$ with compact $n$-skeleton. Let $\widetilde{X}$ be the universal cover of $X$. We say that $H$ is \emph{$(n-1)$-connected at infinity} if for any compact set $C\subseteq \widetilde{X}$ there exists a compact set $C\subseteq D\subseteq \widetilde{X}$ such that the inclusion $\widetilde{X}\setminus D \hookrightarrow \widetilde{X}\setminus C$ induces the zero map in $\pi_k$ for all $k<n$. If a group is $n$-connected at infinity for all $n$ we say it has \emph{trivial homotopy groups at infinity}.
\end{definition}

The property of being $0$-connected at infinity is called being \emph{one-ended}. The property of being $1$-connected at infinity is, naturally, called being \emph{simply connected at infinity}. We also have the obvious parallel notion of having trivial \emph{homology} groups at infinity.

\begin{cit}\cite[Theorem~3.1]{mihalik85}\cite[Theorem~16.9.5]{geoghegan08}\label{cit:hnn_simp_conn_infty}
 Let $B$ be a $1$-ended finitely presented group and $\phi\colon B\to B$ a monomorphism. Then the HNN extension $B*_{\phi,t}$ is simply connected at infinity.
\end{cit}

\begin{corollary}\label{cor:LM_simp_conn_infty}
 The Lodha--Moore groups are simply connected at infinity.
\end{corollary}

\begin{proof}
 The Lodha--Moore groups are clearly $1$-ended, since they are not virtually cyclic and contain no non-abelian free groups. Also, they decompose as ascending HNN extensions of each other as seen in Section~\ref{sec:HNN}, so the result is immediate from Citation~\ref{cit:hnn_simp_conn_infty}.
\end{proof}

Now Theorems~13.3.3(ii) and~17.2.1 of \cite{geoghegan08} reduce our problem to the next proposition. A heuristic summary of this reduction is: since we already have simple connectivity at infinity, the homotopy at infinity of $H$ will vanish as soon as the homology at infinity of $H$ vanishes (thanks to the Hurewicz Theorem), and the homology at infinity of $H$ will vanish as soon as $H^*(H;\Z H)$ vanishes (thanks to some arguments from homological algebra). So, it suffices to prove:

\begin{proposition}\label{prop:LM_htpy_infty_zero}
 For each Lodha--Moore group $H$, we have $H^*(H;\Z H)=0$.
\end{proposition}

\begin{proof}
 We proceed similarly to the proof of \cite[Theorem~7.2]{brown84}; all claims can be compared to the corresponding steps of that proof. We will actually show that for all $i\ge1$, $H^i(H;L)=0$ for any free $\Z H$-module $L$ ($H$ is of type $\F_\infty$ \cite{lodha14} so this is equivalent). We induct on $i$. The base case is that $H^1(H;L)=0$, which follows since $H$ is $1$-ended and finitely generated \cite[Theorem~13.3.3(ii)]{geoghegan08}. Now suppose $i>1$. Let $H=B*_{\phi,t}$ be some HNN decomposition of $H$ from the list (HNN1) through (HNN8). There is a Mayer--Vietoris sequence
 $$\cdots\to H^{i-1}(B;L) \to H^i(H;L) \to H^i(B;L) \stackrel{\phi^*}{\to} H^i(B;L) \to \cdots \text{.}$$
 Since $B$ is itself a Lodha--Moore group, and since $L$ is free over $\Z B$, by induction we are assuming that $H^{i-1}(B;L)=0$. Also, $\phi^*$ is injective since $H$ is of type $\F_\infty$ \cite{brown85}. We conclude from the exactness of the sequence that $H^i(H;L)=0$.
\end{proof}

We conclude:

\begin{theorem}\label{thrm:LM_htpy_infty_zero}
 All the homotopy groups at infinity of any Lodha--Moore group are trivial. \qed
\end{theorem}

As a remark, a crucial point in the proof of the proposition was that $B$ is itself \emph{a} Lodha--Moore group (though not necessarily isomorphic to $H$) and so we can apply the induction hypothesis to $B$.

\medskip

To reiterate, this proves that the Lodha--Moore groups satisfy all four components of the 1979 Geoghegan Conjecture for Thompson's group $F$, namely
\begin{itemize}
 \item[(1)] They are of type $\F_\infty$.
 \item[(2)] They have no non-abelian free subgroups.
 \item[(3)] They are non-amenable.
 \item[(4)] They have trivial homotopy groups at infinity.
\end{itemize}
The first three of these were proved by Lodha \cite{lodha14} and Lodha--Moore \cite{lodha13}. The only part still open for $F$ itself is (3).


\section{The BNS-invariant}\label{sec:inv}

Our HNN decompositions also allow us to analyze the BNSR-invariants of the Lodha--Moore groups. We will first focus on the BNS-invariant $\Sigma^1$, and in Section~\ref{sec:bnsr} we will discuss the higher BNSR-invariants $\Sigma^m$. There are some elegant tools for computing $\Sigma^1$ that have no known generalization for $\Sigma^m$, and in this section we will use these tools to compute $\Sigma^1$ of the Lodha--Moore groups. The work done in Section~\ref{sec:bnsr} to compute $\Sigma^2$ (and parts of the higher $\Sigma^m$) could also be used to compute $\Sigma^1$, but, as the reader will notice, things get very technical in that section, and it is more pleasant to compute $\Sigma^1$ using the aforementioned tools.

The Bieri--Neumann--Strebel (BNS) invariant of a finitely generated group $H$ is a subset of the \emph{character sphere} $S(H)$. A \emph{character} of $H$ is a homomorphism $\chi \colon H\to \R$. (If the image is infinite cyclic, then recall we call this a discrete character.) Two characters are \emph{equivalent} if they differ by multiplication by a positive real number. The equivalence classes $[\chi]$ of characters of $H$ form a sphere $S(H)=S^{d-1}$ where $d$ is the rank of the torsion-free part of the abelianization of $H$. The BNS-invariant $\Sigma^1(H)$ is the subset of $S(H)$ given by:
$$\Sigma^1(H)\defeq \{[\chi]\mid \Gamma(H)^{\chi\ge0}\text{ is connected}\} \text{,}$$
where $\Gamma(H)$ is the Cayley graph of $H$ using any finite generating set, and $\Gamma(H)^{\chi\ge0}$ is the full subgraph spanned by those vertices on which $\chi$ takes non-negative values. It is standard notation to write $\Sigma^1(H)^c$ for $S(H)\setminus \Sigma^1(H)$.

One main application of knowing the BNS-invariant of a group is that it reveals exactly which normal subgroups corresponding to abelian quotients are finitely generated. More precisely, we have:

\begin{cit}\cite[Theorem~1.1]{bieri10}\label{cit:fin_props_bns}
 Let $H$ be a finitely generated group and $N$ a normal subgroup containing $[H,H]$. Then $N$ is finitely generated if and only if $[\chi]\in\Sigma^1(H)$ for every character $\chi$ with $\chi(N)=0$.
\end{cit}

\subsection{Tools}\label{sec:tools}

We collect here a variety of tools for deducing whether characters are in the BNS-invariant $\Sigma^1$ or its complement. Some generalize to higher $\Sigma^m$ and some do not. The first two results involve ascending HNN extensions. They deal with the cases of the base group either lying in the kernel of a character, or not. These both generalize to higher $\Sigma^m$, as we will mention in Section~\ref{sec:bnsr}.

\begin{cit}\cite[Theorem~2.1]{bieri10}\label{cit:hnn_bns}
 Let $H=B*_t$ for $B$ finitely generated, and let $\chi \colon H\to\R$ be the character given by $\chi(B)=0$ and $\chi(t)=1$. Then $[\chi]\in\Sigma^1(H)$ and moreover if $B^t\subsetneq B$ then $[-\chi]\in\Sigma^1(H)^c$.
\end{cit}

\begin{cit}\cite[Theorem~2.3]{bieri10}\label{cit:push_Sigma_1}
 Let $H=B*_t$ for $B$ finitely generated, and let $\chi\colon H\to \R$ be a character. Suppose $\chi|_B \neq 0$ and $[\chi|_B]\in \Sigma^1(B)$. Then $[\chi]\in\Sigma^1(H)$.
\end{cit}

It is in general very difficult to compute the BNS-invariant using the definition. An alternate characterization of $\Sigma^1(H)$, due to Brown \cite{brown87bns}, involves inspecting actions of $H$ on $\R$-trees, and over the years this definition has led to some powerful technology for computing $\Sigma^1(H)$. A modern formulation is distilled in \cite{koban14}, and we review the details here. Analogues of this technique for higher $\Sigma^m$ have yet to be developed.

Given a character $\chi$ of a group $H$, we call an element $h\in H$ \emph{$\chi$-hyperbolic} if $\chi(h)\neq0$. By the \emph{commuting graph} $C(J)$ of a subset $J\subseteq H$ we mean the graph whose vertex set is $J$ and which has an edge connecting $g,h\in J$ if and only if $[g,h]=1$. Finally, given $I,J\subseteq H$, we say $J$ \emph{dominates} $I$ if every element of $I$ commutes with some element of $J$.

\begin{cit}[Hyperbolics dominating generators]\cite[Lemma~1.9]{koban14}\label{cit:conn_dom}
 Let $\chi$ be a character of a group $H$. If there exists a set $J$ of $\chi$-hyperbolic elements with $C(J)$ connected, and a set $I$ generating $H$ such that $J$ dominates $I$, then $[\chi]\in\Sigma^1(H)$.
\end{cit}

In summary, this provides a nice way to translate knowledge about generators and commutator relations into knowledge about the BNS-invariant. In the next subsection we compile all these tools together, ultimately computing the BNS-invariant for the Lodha--Moore groups.

\subsection{Computations}\label{sec:computations}

The main result of this subsection is:

\begin{theorem}\label{thrm:BNS}
 The BNS-invariants for $\LMsmall$, $\LMright$, $\LMleft$ and $\LMbig$ are all of the form $S^2 \setminus P$ for some set $P$ with $|P|=2$. More precisely, for $\LMsmall$ we have $P=\{[\chi_0],[\chi_1]\}$, for $\LMleft$ we have $P=\{[\psi_0],[\chi_1]\}$, for $\LMright$ we have $P=\{[\chi_0],[-\psi_1]\}$ and for $\LMbig$ we have $P=\{[\psi_0],[-\psi_1]\}$.
\end{theorem}

We first focus on $\LMbig$. Recall that every character $\chi$ is of the form $\chi=a\psi_0+b\psi_1+c\psi$ for $a,b,c\in\R$.

\begin{proposition}[North/south hemispheres]\label{prop:big_hemispheres}
 Let $\chi$ be a character of $\LMbig$ as above with $c\neq 0$. Then $[\chi]\in\Sigma^1(\LMbig)$.
\end{proposition}

\begin{proof}
 We will apply Citation~\ref{cit:conn_dom}. Let $J\defeq \{y_s \mid s\neq 0^n,1^n\}$, so every element of $J$ is $\psi$-hyperbolic. Also, every element of $J$ is in the kernels of $\psi_0$ and $\psi_1$, so in fact they are all $\chi$-hyperbolic. Note that $[y_{0t},y_{1u}]=1$ for any $t,u$, so $C(J)$ is connected. It remains to show that some generating set of $\LMbig$ is dominated by $J$, and we will use the generators $x_\emptyset x_1^{-1},x_1,y_0,y_{10},y_1$. These generators respectively commute with the $\chi$-hyperbolic elements $y_s$ for $s=110$, $01$, $10$, $01$ and $01$.
\end{proof}

\begin{proposition}[Most of equator]\label{prop:big_equator}
 Let $\chi = a\psi_0+b\psi_1+c\psi$ be a character of $\LMbig$ with $c=0$, and $a\neq0$ and $b\neq0$. Then $[\chi]\in\Sigma^1(\LMbig)$.
\end{proposition}

\begin{proof}
 Let $J\defeq \{y_{0^n},y_{1^n} \mid n>0\}$, so every element is $\chi$-hyperbolic. Every $y_{0^n}$ commutes with every $y_{1^m}$ for $n,m>0$, so $C(J)$ is connected. The generators $x_\emptyset x_1^{-1},x_1,y_0,y_{10},y_1$ respectively commute with the $\chi$-hyperbolic elements $y_s$ for $s=11$, $0$, $1$, $0$ and $0$.
\end{proof}

\begin{lemma}[Remaining four points]\label{lem:big_poles}
 We have $[-\psi_0],[\psi_1]\in\Sigma^1(\LMbig)$ and $[\psi_0],[-\psi_1] \in \Sigma^1(\LMbig)^c$.
\end{lemma}

\begin{proof}
 From (HNN7) and (HNN8) we know that $\LMbig = (\LMright)*_{y_0^{-1}}$ and $\LMbig = (\LMleft)*_{y_1}$. Note that $-\psi_0(\LMright)=0$ and $-\psi_0(y_0^{-1})=1$. Similarly $\psi_1(\LMleft)=0$ and $\psi_1(y_1)=1$. The result now follows from Citation~\ref{cit:hnn_bns}.
\end{proof}

These three results combine to yield:

\begin{corollary}
 $\Sigma^1(\LMbig)^c = \{[\psi_0],[-\psi_1]\}$. \qed
\end{corollary}

\medskip

Next we consider $\LMright$. The cases proceed essentially the same as for $\LMbig$, except now the first case is shorter, thanks to having already handled $\LMbig$.

\begin{lemma}[North/south hemispheres]\label{lem:right_hemispheres}
 Let $\chi=a\chi_0+b\psi_1+c\psi$ be a character of $\LMright$ with $c\neq0$. Then $[\chi] \in \Sigma^1(\LMright)$.
\end{lemma}

\begin{proof}
 From (HNN3) we have $\LMright = (\LMbig)*_{x_\emptyset}$. The restriction of $\chi$ to $\LMbig$ has non-zero $\psi$ coefficient. The result now follows from Proposition~\ref{prop:big_hemispheres} and Citation~\ref{cit:push_Sigma_1}.
\end{proof}

\begin{proposition}[Most of equator]\label{prop:right_equator}
 Let $\chi = a\chi_0+b\psi_1+c\psi$ be a character of $\LMright$ with $c=0$, and $a\neq0$ and $b\neq0$. Then $[\chi]\in\Sigma^1(\LMright)$.
\end{proposition}

\begin{proof}
 We apply Citation~\ref{cit:conn_dom}. Let $J\defeq \{x_{0^n},y_{1^n} \mid n>0\}$, so every element is $\chi$-hyperbolic. Every $x_{0^n}$ commutes with every $y_{1^m}$ for $n,m>0$, so $C(J)$ is connected. The generators $x_\emptyset x_1^{-1},x_1,y_0,y_{10},y_1$ respectively commute with the $\chi$-hyperbolic elements $y_{11}$, $x_0$, $y_1$, $x_0$ and $x_0$.
\end{proof}

\begin{lemma}[Remaining four points]\label{lem:right_poles}
 We have $[-\chi_0],[\psi_1]\in\Sigma^1(\LMright)$ and $[\chi_0],[-\psi_1]\in\Sigma^1(\LMright)^c$.
\end{lemma}

\begin{proof}
 By (HNN3) and (HNN4) we know that $\LMright = (\LMbig)*_{x_\emptyset}$ and $\LMright = \LMsmall*_{y_1}$. Note that $-\chi_0(\LMbig)=0$ (since here $\LMbig$ means $\LMright(1)$) and $-\chi_0(x_\emptyset)=1$. Similarly $\psi_1(\LMsmall)=0$ and $\psi_1(y_1)=1$. The result now follows from Citation~\ref{cit:hnn_bns}.
\end{proof}

\begin{corollary}
 $\Sigma^1(\LMright)^c = \{[\chi_0],[-\psi_1]\}$. \qed
\end{corollary}

\medskip

Parallel arguments applied to $\LMleft$ and $\LMsmall$ now quickly give us:

\begin{corollary}
 $\Sigma^1(\LMleft)^c = \{[\psi_0],[\chi_1]\}$ and $\Sigma^1(\LMsmall)^c = \{[\chi_0],[\chi_1]\}$. \qed
\end{corollary}

This finishes the proof of Theorem~\ref{thrm:BNS}.

Note that $\Sigma^1(F)=S(F)\setminus\{[\chi_0],[\chi_1]\}$, so for $\LMsmall$ we now know that $[\chi]\in\Sigma^1(\LMsmall)$ if and only if $\chi(F)=0$ or else $[\chi|_F]\in\Sigma^1(F)$. We can phrase this as, the inclusion $F\hookrightarrow\LMsmall$ induces an isomorphism in $\Sigma^1(-)^c$.

\section{The BNSR-invariants}\label{sec:bnsr}

Bieri and Renz \cite{bieri88} extended the invariant $\Sigma^1$ to a family of invariants $\Sigma^m$ for $m\in\N$. These form a nested chain of subsets of the character sphere $S(H)$ of a group $H$, namely
$$S(H)\supseteq \Sigma^1(H) \supseteq \Sigma^2(H) \supseteq \cdots$$
with $\Sigma^\infty(H) \defeq \bigcap\limits_{m\in\N}\Sigma^m(H)$. The invariants $\Sigma^m(H)$ are called the \emph{Bieri--Neumann-Strebel--Renz (BNSR) invariants} of $H$.

In this section we show that a full computation of all the $\Sigma^m$ for the Lodha--Moore groups would follow from a single conjecture, namely that $\ker(\psi)$ is of type $\F_\infty$. There is evidence that this is true, but it appears to be significantly more difficult than proving that the Lodha--Moore groups themselves are of type $\F_\infty$, which was already quite difficult in \cite{lodha14}. We can at least prove $\ker(\psi)$ is finitely presented; this allows us to fully compute $\Sigma^2$ of all the Lodha--Moore groups.

We will make almost no use of the definition itself, but we state it here for completeness; see \cite[Section~1.2]{bieri10}.

\begin{definition}[BNSR-invariants]\label{def:bnsr}
 Let $H$ be a group with finite presentation $\langle S\mid R\rangle$. Let $\Gamma$ be the Cayley graph of $H$ with respect to $S$. Pick an $H$-invariant orientation for each edge and glue in a $2$-cell for each relation in $R$, equivariantly along the appropriate loops in the graph. Call the resulting $2$-complex $\Gamma^2$. For a character $\chi$ of $H$ and $n\in\Z$, let $\Gamma^2_{\chi\ge n}$ be the full subcomplex spanned by vertices $h$ with $\chi(h)\ge n$. Now the definition of $\Sigma^2(H)$ is that $[\chi]\in\Sigma^2(H)$ if $\Gamma^2_{\chi\ge 0}$ is connected and there exists $n\le0$ such that the inclusion $\Gamma^2_{\chi\ge 0} \subseteq \Gamma^2_{\chi\ge n}$ induces the trivial map in $\pi_1$.

 For $H$ of type $\F_m$, the definition of $\Sigma^m(H)$ is similar. Let $\Gamma^m$ be the $m$-skeleton of the universal cover of a $K(H,1)$ with finite $m$-skeleton. Then $[\chi]$ is in $\Sigma^m(H)$ if and only if the above inclusion (for some $n$) induces the trivial map in $\pi_k$ for all $k<m$.
\end{definition}

One main application of the BNSR-invariants is the following:

\begin{cit}\cite[Theorem~1.1]{bieri10}\label{cit:bnsr_fin_props}
 Let $H$ be a group of type $\F_m$ and $[H,H]\le N\triangleleft H$. Then $N$ is of type $\F_m$ if and only if for every non-zero character $\chi$ of $H$ with $\chi(N)=0$, we have $[\chi]\in\Sigma^m(H)$.
\end{cit}

In particular the BNSR-invariants of a group of type $\F_\infty$ give a complete catalog of exactly which normal subgroups containing the commutator subgroup have which finiteness properties.

\medskip

The Lodha--Moore groups are all of type $\F_\infty$. Lodha has proved this in a recent preprint on arXiv \cite{lodha14} for $\LMsmall$, and his proof works in parallel for $\LMright$, $\LMleft$ and $\LMbig$. The idea is to construct a contractible space $X$ on which $\LMsmall$ acts with stabilizers of type $\F_\infty$ and with finitely many orbits of cells in each dimension, after which it follows from classical results that $\LMsmall$ is of type $\F_\infty$. The ``hard part'' is proving that $X$ is contractible.

Since we already know that the Lodha--Moore groups are of type $\F_\infty$, and since we have shown earlier that they decompose as nice HNN extensions, it is not too hard to compute a large piece of the BNSR-invariants. The key tools are the following generalizations of Citations~\ref{cit:hnn_bns} and~\ref{cit:push_Sigma_1}.

\begin{cit}\cite[Theorem~2.1]{bieri10}\label{cit:hnn_bnsr}
 Let $H=B*_t$ for $B$ of type $\F_m$ ($m\in\N\cup\{\infty\}$), and let $\chi \colon H\to\R$ be the character given by $\chi(B)=0$ and $\chi(t)=1$. Then $[\chi]\in\Sigma^m(H)$.
\end{cit}

\begin{lemma}\label{lem:push_Sigma_m}
 Let $H$ be a group that is an ascending HNN extension $H=B\ast_t$ for $B$ a subgroup of type $\F_m$ ($m\in\N\cup\{\infty\}$). Let $\chi\colon H\to \R$ be a character. Suppose $\chi|_B \neq 0$ and $[\chi|_B]\in \Sigma^m(B)$. Then $[\chi]\in\Sigma^m(H)$.
\end{lemma}

\begin{proof}
 This follows in the same way as for the $m=\infty$ case in \cite[Theorem~2.3]{bieri10}, namely by combining \cite[Proposition~4.1]{meinert96} with \cite[Theorem~B]{meinert97} (and observing that for $m\ge2$, $\Sigma^m(H)=\Sigma^2(H)\cap\Sigma^m(H,\Z)$ for any $H$).
\end{proof}

If the following conjecture holds for all $m\in\N\cup\{\infty\}$, then, as we will soon show, we have a full computation of the BNSR-invariants for all four Lodha--Moore groups $\LMsmall$, $\LMright$, $\LMleft$, $\LMbig$.

\begin{conjecture}[$\textrm{C}_m$]\label{conj:ker_psi_F_m}
 The kernel of $\psi$ in some Lodha--Moore group is of type $\F_m$.
\end{conjecture}

Note that for $m\ge\ell$, ($\textrm{C}_m$) implies ($\textrm{C}_\ell$). We have already seen that ($\textrm{C}_1$) is true, since $[\pm\psi]\in\Sigma^1(\LMsmall)$. We have evidence to suggest that ($\textrm{C}_m$) holds for all $m$ (meaning that ($\textrm{C}_\infty$) holds), but verifying it has proved to be significantly more difficult than Lodha's proof that $\LMsmall$ is of type $\F_\infty$, and this was already quite difficult. As such, for the rest of the section we leave these as conjectures and show how, for any $m$, if ($\textrm{C}_m$) is true then the full computation of $\Sigma^m$ falls out. Then at the end we will verify ($\textrm{C}_2$) ``by hand'' and so fully compute $\Sigma^2$ for all four Lodha--Moore groups.

\medskip

\begin{observation}\label{obs:one_for_all}
 For any $m$, if ($\textrm{C}_m$) holds for some Lodha--Moore group $\LMsmall$, $\LMright$, $\LMleft$, $\LMbig$, then it holds for all of them.
\end{observation}

\begin{proof}
 For any two Lodha--Moore groups $H$ and $H'$, (HNN1) through (HNN8) tell us that $H'$ is either an ascending HNN extension of $H$ or else an ascending HNN extension of an ascending HNN extension of $H$. (For example $\LMbig=(\LMsmall*_{y_0^{-1}})*_{y_1}$.) For any such $H=B*_t$, $\pm\psi|_B=\pm\psi$. By Citation~\ref{cit:bnsr_fin_props}, $\ker(\psi)$ in some group is of type $\F_m$ if and only if $[\pm\psi]$ is in $\Sigma^m$ of that group. The result now follows from repeated applications of Lemma~\ref{lem:push_Sigma_m}.
\end{proof}

\begin{corollary}\label{cor:psi_in_Sigma_infty}
 Assume ($\textrm{C}_m$) is true. Then for any Lodha--Moore group $H$, $[\pm\psi]\in\Sigma^m(H)$.
\end{corollary}

\begin{proof}
 This is immediate from Observation~\ref{obs:one_for_all} and Citation~\ref{cit:bnsr_fin_props}.
\end{proof}

\begin{lemma}[Non-zero $\psi$ coefficient]\label{lem:small_hemispheres_bnsr}
 Suppose ($\textrm{C}_m$) holds. Let $\chi$ be a character of a Lodha--Moore group $H$, with non-zero $\psi$ coefficient. Then $[\chi]\in\Sigma^m(H)$.
\end{lemma}

\begin{proof}
 In the first case we consider, suppose that if $H=\LMsmall$ or $\LMleft$ then $\chi$ has $\chi_1$ coefficient zero and if $H=\LMright$ or $\LMbig$ then $\chi$ has $\psi_1$ coefficient zero (in other words, whichever ``basis character on the right'' is defined and non-zero for $H$, in this case we suppose that coefficient is zero). For whichever $H$ we have, let $B$ and $t$ be such that $H=B*_t$ is the relevant one of (HNN1), (HNN3), (HNN5) or (HNN7). Since $\chi$ has non-zero $\psi$ coefficient, in all four cases $\chi|_B\neq0$, and thanks to our assumption in fact $\chi|_B = a\psi$ for some $a\in\R^\times$. But $[a\psi]\in\Sigma^m(B)$ by ($\textrm{C}_m$), Citation~\ref{cit:bnsr_fin_props} and Observation~\ref{obs:one_for_all}. Hence by Lemma~\ref{lem:push_Sigma_m} we conclude $[\chi]\in\Sigma^m(H)$.
 
 Now suppose $\chi$ has non-zero $\chi_1$ or $\psi_1$ coefficient, depending on which is defined and non-zero for our $H$. Again we consider HNN decompositions, this time (HNN2), (HNN4), (HNN6) and (HNN8); let $B$ and $t$ be such that $H=B*_t$ appears on this list. The restriction of $\chi$ to $B$ now has $\chi_1$ or $\psi_1$ coefficient zero, so by the previous paragraph $[\chi|_B]\in\Sigma^m(B)$ and hence $[\chi]\in\Sigma^m(\LMsmall)$.
\end{proof}

Next we focus on the case when $\chi$ has $\psi$ coefficient zero. (At this point we do not need to care whether ($\textrm{C}_m$) is true.) Since $\Sigma^1(H)^c\subseteq \Sigma^m(H)^c$ for any group $H$, we already know some points that are not in $\Sigma^m$ of our groups; for example $[\chi_0]\in\Sigma^m(\LMsmall)^c$, $[-\psi_1]\in\Sigma^m(\LMright)^c$, etc. The next lemma looks at character classes of the form $[\pm\chi_i]$ and $[\pm\psi_i]$ and shows that they are in $\Sigma^m$ if and only if they are in $\Sigma^1$.

\begin{lemma}\label{lem:poles_bnsr}
 We have $[-\chi_0],[-\chi_1]\in\Sigma^m(\LMsmall)$, $[-\chi_0],[\psi_1]\in\Sigma^m(\LMright)$, $[-\psi_0],[-\chi_1]\in\Sigma^m(\LMleft)$ and $[-\psi_0],[\psi_1]\in\Sigma^m(\LMbig)$.
\end{lemma}

\begin{proof}
 This is just a matter of applying Citation~\ref{cit:hnn_bnsr} to the right HNN decompositions. For example, (HNN1), which is $\LMsmall=(\LMleft)*_{x_\emptyset}$, shows $[-\chi_0]\in\Sigma^m(\LMsmall)$, since $-\chi_0(\LMleft)=0$ (recall here $\LMleft$ means $\LMsmall(1)$) and $-\chi_0(x_\emptyset)=1$. As another example, (HNN8), which is $\LMbig=\LMleft*_{y_1}$, shows that $[\psi_1]\in\Sigma^m(\LMbig)$, since $\psi_1(\LMleft)=0$ and $\psi_1(y_1)=1$. The other cases all follow similarly easily.
\end{proof}

With the multiples of the basis characters fully understood, we can now take care of the remaining case, when the $\psi$ coefficient is zero but the other two are not. It turns out that the results mirror the situation for $F$ done in \cite{bieri10}. First we handle ``three quarters'' of this situation (compare to Corollary~2.4 in \cite{bieri10}).

\begin{lemma}\label{lem:long_interval}
 Let $H\in\{\LMsmall,\LMright,\LMleft,\LMbig\}$. If $H=\LMsmall$ or $\LMright$ let $\xi_0=\chi_0$ and if $H=\LMleft$ or $\LMbig$ let $\xi_0=\psi_0$. Similarly let $\xi_1$ be whichever of $\chi_1$ or $-\psi_1$ is defined and non-zero on $H$. Let $\chi=a\xi_0+b\xi_1$ be a character of $H$ with $\psi$ coefficient zero, and with $a,b\neq0$. If $a<0$ or $b<0$ then $[\chi]\in\Sigma^\infty(H)$.
\end{lemma}

\begin{proof}
 First suppose $a<0$. Let $B$ and $t$ be such that the expression $H=B*_t$ is the relevant one of (HNN2), (HNN4), (HNN6) or (HNN8). In all cases, $\xi_1(B)=0$, so $[\chi|_B]=[-\xi_0]$, which is in $\Sigma^\infty(B)$ by Lemma~\ref{lem:poles_bnsr}, and so by Lemma~\ref{lem:push_Sigma_m} we conclude that $[\chi]\in\Sigma^\infty(H)$.

 Now suppose $b<0$. This time let $B$ and $t$ be such that $H=B*_t$ is the relevant one of (HNN1), (HNN3), (HNN5) or (HNN7). Then $\xi_0(B)=0$ so $[\chi|_B]=[b\xi_1]$. This is either $[-\chi_1]$ or $[\psi_1]$, which in either case is in $\Sigma^\infty(B)$. Hence by Lemma~\ref{lem:push_Sigma_m}, $[\chi]\in\Sigma^\infty(H)$.
\end{proof}

Lastly we handle the remaining ``one quarter'' of this situation. We find that in this case, even though $[\chi]\in\Sigma^1(H)$, in fact $[\chi]\in\Sigma^2(H)^c$, so of course $[\chi]\in\Sigma^\infty(H)^c$. We could even show the stronger fact that $[\chi]\in\Sigma^2(H,R)^c$ for any $R$, but we have not (and will not) define the homological invariants $\Sigma^m(H,R)$, and proving this would require a long digression introducing them, plus a couple long proofs, all of which would be simple imitations of those found in Section~2.3 of \cite{bieri10}. As such, we will only cover the homotopical case here, for which the proof amounts to a citation.

\begin{observation}
 Let $H\in\{\LMsmall,\LMright,\LMleft,\LMbig\}$. If $H=\LMsmall$ or $\LMright$ let $\xi_0=\chi_0$ and if $H=\LMleft$ or $\LMbig$ let $\xi_0=\psi_0$. Similarly let $\xi_1$ be whichever of $\chi_1$ or $-\psi_1$ is defined and non-zero on $H$. Let $\chi=a\xi_0+b\xi_1$ be a character of $H$ with $\psi$ coefficient zero. Suppose that $a>0$ and $b>0$. Then $[\chi]\in\Sigma^2(H)^c$.
\end{observation}

\begin{proof}
 An equivalent way to phrase this is to say that the convex hull in $S(H)$ of the two points in $\Sigma^1(H)^c$ lies in $\Sigma^2(H)^c$. But this is immediate from \cite[Theorem~2.6]{bieri10}, since the Lodha--Moore groups are all finitely presented and contain no non-abelian free subgroups.
\end{proof}

In summary, for any $m\in\N\cup\{\infty\}$, assuming the truth of ($\textrm{C}_m$), we have fully computed $\Sigma^m(H)$ for all four Lodha--Moore groups $H$. Namely, ($\textrm{C}_1$) is true and $\Sigma^1(H)$ is computed in Theorem~\ref{thrm:BNS}, and:

\begin{theorem}\label{thrm:BNSR}
 Let $m>1$, allowing for $m=\infty$. If ($\textrm{C}_m$) is true, then for each Lodha--Moore group $H$, we have that $\Sigma^m(H)$ equals $\Sigma^1(H)$ with the convex hull of $\Sigma^1(H)^c$ removed. \qed
\end{theorem}

\subsection{The case of $\Sigma^2$}\label{sec:Sigma^2}

It turns out we can verify the conjecture ($\textrm{C}_2$) ``by hand'' and hence compute $\Sigma^2(H)$ for any Lodha--Moore group $H$. We need to show that $\ker(\psi)$ in some $H$ is finitely presented. We will do this for $H=\LMsmall$, i.e., we will show that $K\defeq\ker(\psi)\le\LMsmall$ is finitely presented. The starting point is the finite presentation for $\LMsmall$ (there denoted $G_0$) at the end of \cite[Section~3]{lodha13}. This presentation has three generators,
$$a=x_\emptyset \text{, } b=x_1 \text{ and } c=y_{10} \text{,}$$
and nine relations,
\begin{itemize}
 \item[($\textrm{R1}$)] $ba^{-2}ba^2 b^{-1}a^{-1}b^{-1}a = 1$
 \item[($\textrm{R2}$)] $ba^{-3}ba^3 b^{-1}a^{-2}b^{-1}a^2 = 1$
 \item[($\textrm{R3}$)] $ca^2 b^{-1}a^{-1}c^{-1}aba^{-2} = 1$
 \item[($\textrm{R4}$)] $ab^2 a^{-1}b^{-1}ab^{-1}a^{-1}caba^{-1}bab^{-2}a^{-1}c^{-1} = 1$
 \item[($\textrm{R5}$)] $ca^{-1}bac^{-1}a^{-1}b^{-1}a = 1$
 \item[($\textrm{R6}$)] $ca^{-2}ba^2 c^{-1}a^{-2}b^{-1}a^2 = 1$
 \item[($\textrm{R7}$)] $caca^{-1}c^{-1}ac^{-1}a^{-1} = 1$
 \item[($\textrm{R8}$)] $ca^2 ca^{-2}c^{-1}a^2 c^{-1}a^{-2} = 1$
 \item[($\textrm{R9}$)] $b^2 a^{-1}b^{-1}aca^{-1}bc^{-1}a^{-1}cab^{-1}ab^{-1}c^{-1} = 1$
\end{itemize}
(Note that ($\textrm{R4}$) and ($\textrm{R9}$) are not written right in \cite{lodha13} (v3 on arXiv); as written here they are the correct translations into $a,b,c$ of the relations $[y_{10},x_{01}]=1$ and $y_{10} = x_{10} y_{100} y_{1010}^{-1} y_{1011}$.)

Instead of using $b=x_1$, we first rephrase everything using $d\defeq x_0$. We do this because $d$ and $c$ commute, and this ends up making everything that follows more elegant. We have $d=a^2 b^{-1}a^{-1}$ and $b=a^{-1}d^{-1}a^2$, so a finite presentation for $\LMsmall$ with generating set $a,d,c$ has the nine relations:
\begin{itemize}
 \item[($\textrm{R1}'$)] $a^{-1}d^{-1}a^{-1}d^{-1}a^2 da^{-2}da^2 = 1$
 \item[($\textrm{R2}'$)] $a^{-1}d^{-1}a^{-2}d^{-1}a^3 da^{-3}da^3 = 1$
 \item[($\textrm{R3}'$)] $cdc^{-1}d^{-1} = 1$
 \item[($\textrm{R4}'$)] $d^{-1}ad^{-1}a^{-1}d^2 cd^{-2}ada^{-1}dc^{-1} = 1$
 \item[($\textrm{R5}'$)] $ca^{-2}d^{-1}a^3 c^{-1}a^{-3}da^2 = 1$
 \item[($\textrm{R6}'$)] $ca^{-3}d^{-1}a^4 c^{-1}a^{-4}da^3 = 1$
 \item[($\textrm{R7}'$)] $caca^{-1}c^{-1}ac^{-1}a^{-1} = 1$
 \item[($\textrm{R8}'$)] $ca^2 ca^{-2}c^{-1}a^2 c^{-1}a^{-2} = 1$
 \item[($\textrm{R9}'$)] $a^{-1}d^{-1}ad^{-1}a^{-1}da^2 ca^{-2}d^{-1}a^2c^{-1}a^{-1}ca^{-1}d^2 ac^{-1} = 1$
\end{itemize}

Let $R$ denote the set of these nine relations. Let $X$ be the presentation $2$-complex for this presentation, so $\pi_1(X)\cong \LMsmall$. Let $Y\to X$ be the cover with $\pi_1(Y)\cong K$. The $1$-skeleton of $Y$ consists of vertices $v_n$ for $n\in\Z$ and edges $a_n,d_n,c_n$ for each $n\in\Z$, where $a_n$ and $d_n$ are loops based at $v_n$, and $c_n$ goes from $v_n$ to $v_{n+1}$. Write $c_n^{-1}$ for the opposite orientation of $c_n$. Fix $v_0$ as a basepoint. The fundamental group of $Y$ is clearly generated by edge paths of the form
\begin{align*}
 & c_0,c_1\dots,c_n,a_n,c_n^{-1},\dots,c_0^{-1} \text{, }\\
 & c_0,c_1\dots,c_n,d_n,c_n^{-1},\dots,c_0^{-1} \text{, }\\
 & c_{-1}^{-1}\dots,c_{-n}^{-1},a_{-n},c_{-n},\dots,c_{-1} \text{ and }\\
 & c_{-1}^{-1}\dots,c_{-n}^{-1},d_{-n},c_{-n},\dots,c_{-1}
\end{align*}
for $n\ge0$. By slight abuse of notation we will call these $a_n$ and $d_n$, so for example $a_n\in\pi_1(Y)$ is the loop that goes from $v_0$ to $v_n$ along $c$-edges, then loops around $a_n$, and then comes back to $v_0$ along $c$-edges.

We now have an infinite generating set for $K$, namely
$$\{a_n,d_n\mid n\in\Z\}\text{.}$$
In fact this is just the generating set obtained for $K$ from Schreier's Lemma, using the generating set $\{a,d,c\}$ and the transversal $\{c^n\mid n\in\Z\}$ for $K\hookrightarrow\langle a,d,c\rangle\onto\langle c\rangle$.

We also know something about the $2$-cells of $Y$. For any $(w=1)\in R$, and every $n\in\Z$, let $w_n$ be the path in $Y^{(1)}$ traversed by reading $w$ starting at $v_n$. For example, if $w=ca^{-2}d^{-1}a^3c^{-1}a^{-3}da^2$ is from ($\textrm{R5}'$), then $w_n=a_{n+1}^{-2}d_{n+1}^{-1}a_{n+1}^3 a_n^{-3}d_n a_n^2$. These $w_n$ are loops since the net sum of exponents of $c$ for any $w$ is zero (in other words $\psi(w)=0$). Then $Y$ is obtained from $Y^{(1)}$ by attaching a $2$-cell along each boundary $w_n$.

\medskip

Since $\pi_1(Y)\cong K$, at this point we have an infinite presentation for $K$. The generators are the $a_n,d_n$ for $n\in\Z$, and the relations are $w_n=1$ for $w\in R$ and $n\in\Z$. Hence our presentation for $K$ has the nine infinite families of relations:
\begin{itemize}
 \item[($\textrm{K1}_n'$)] $a_n^{-1}d_n^{-1}a_n^{-1}d_n^{-1}a_n^2 d_n a_n^{-2}d_n a_n^2 = 1$
 \item[($\textrm{K2}_n'$)] $a_n^{-1}d_n^{-1}a_n^{-2}d_n^{-1}a_n^3 d_n a_n^{-3}d_n a_n^3 = 1$
 \item[($\textrm{K3}_n'$)] $d_{n+1}d_n^{-1} = 1$
 \item[($\textrm{K4}_n'$)] $d_n^{-1}a_n d_n^{-1}a_n^{-1}d_n^2 d_{n+1}^{-2}a_{n+1}d_{n+1}a_{n+1}^{-1}d_{n+1} = 1$
 \item[($\textrm{K5}_n'$)] $a_{n+1}^{-2}d_{n+1}^{-1}a_{n+1}^3 a_n^{-3}d_n a_n^2 = 1$
 \item[($\textrm{K6}_n'$)] $a_{n+1}^{-3}d_{n+1}^{-1}a_{n+1}^4 a_n^{-4}d_n a_n^3 = 1$
 \item[($\textrm{K7}_n'$)] $a_{n+1}a_{n+2}^{-1}a_{n+1}a_n^{-1} = 1$
 \item[($\textrm{K8}_n'$)] $a_{n+1}^2 a_{n+2}^{-2}a_{n+1}^2 a_n^{-2} = 1$
 \item[($\textrm{K9}_n'$)] $a_n^{-1}d_n^{-1}a_n d_n^{-1}a_n^{-1}d_n a_n^2 a_{n+1}^{-2}d_{n+1}^{-1}a_{n+1}^2 a_n^{-1}a_{n+1}^{-1}d_{n+1}^2 a_{n+1} = 1$
\end{itemize}

Our first goal is to reduce this presentation to one with finitely many generators. Define
$$z\defeq a_0^{-1}a_1 \text{.}$$

\begin{lemma}[Finite generating set]\label{lem:K_fg}
 For any $n\in\Z$ we have $a_n=a_0 z^n$ and $d_n=d_0$.
\end{lemma}

\begin{proof}
 By definition, $z^n = (a_0^{-1}a_1)^n$. By the relations ($\textrm{K7}_n'$), $a_n^{-1}a_{n+1} = a_{n+1}^{-1}a_{n+2}$ for all $n\in\Z$. Hence for $n\ge0$ we have
 $$z^n = (a_0^{-1}a_1)(a_1^{-1}a_2)\cdots(a_{n-1}^{-1}a_n) = a_0^{-1}a_n$$
 and for $n<0$ we have
 $$z^n=(a_0^{-1}a_{-1})(a_{-1}^{-1}a_{-2})\cdots(a_{n+1}^{-1}a_n) = a_0^{-1}a_n \text{.}$$
 In either case, $a_n=a_0 z^n$.
 
 The fact that $d_n=d_0$ for all $n$ is just the content of the relations ($\textrm{K3}_n'$).
\end{proof}

Setting $a=a_0$ and $d=d_0$ (which were their names in $\LMsmall$ anyway), we can now convert our presentation for $K$ into one using just the generators $a$, $d$ and $z$. After the substitutions $a_n=az^n$ and $d_n=d$, and some free cyclic reductions, our nine families of relations become:
\begin{itemize}
 \item[($\textrm{K1}_n$)] $d^{-1}z^{-n}a^{-1}d^{-1}az^n az^n dz^{-n}a^{-1}z^{-n}a^{-1}daz^n = 1$
 \item[($\textrm{K2}_n$)] $d^{-1}z^{-n}a^{-1}z^{-n}a^{-1}d^{-1}az^n az^n az^n dz^{-n}a^{-1}z^{-n}a^{-1}z^{-n}a^{-1}daz^n az^n = 1$
 \item[($\textrm{K3}_n$)] $1=1$
 \item[($\textrm{K4}_n$)] $d^{-1}zdz^{-1} = 1$
 \item[($\textrm{K5}_n$)] $z^{-1}a^{-1}z^{-(n+1)}a^{-1}d^{-1}az^{n+1}az^{n+1}aza^{-1}z^{-n}a^{-1}z^{-n}a^{-1}daz^n a = 1$
 \item[($\textrm{K6}_n$)] $z^{-1}a^{-1}z^{-(n+1)}a^{-1}z^{-(n+1)}a^{-1}d^{-1}az^{n+1}az^{n+1}az^{n+1}aza^{-1}z^{-n}a^{-1}z^{-n}a^{-1} z^{-n}a^{-1}daz^n az^n a = 1$
 \item[($\textrm{K7}_n$)] $1 = 1$
 \item[($\textrm{K8}_n$)] $zaz^{-1}a^{-1}z^{-1}aza^{-1} = 1$
 \item[($\textrm{K9}_n$)] $a^{-1}d^{-1}az^n d^{-1}z^{-n}a^{-1}daz^n az^{-1}a^{-1}z^{-(n+1)}a^{-1}d^{-1}az^{n+1}aza^{-1}z^{-(n+1)}a^{-1}d^2az = 1$
\end{itemize}
Note that ($\textrm{K3}_n$) and ($\textrm{K7}_n$) are trivial, and that ($\textrm{K4}_n$) and ($\textrm{K8}_n$) are independent of $n$, so we may rename them ($\textrm{K4}$) and ($\textrm{K8}$). We can write them in the following more concise forms:
\begin{itemize}
 \item[($\textrm{K4}$)] $[z,d]=1$.
 \item[($\textrm{K8}$)] $[z,aza^{-1}]=1$.
\end{itemize}

\medskip

To recap, we have $K\cong\langle a,d,z\mid (\textrm{K1}_n) \text{ through }(\textrm{K9}_n) \text{ hold for all }n\rangle$ and we would like to find a finite presentation.

We introduce three additional relations that will prove useful.
\begin{itemize}
 \item[($\textrm{K10}$)] $[z,ada^{-1}]=1$
 \item[($\textrm{K11}$)] $[z,a^2 da^{-2}]=1$
 \item[($\textrm{K12}$)] $[z,a^2 za^{-2}]=1$
 \item[($\textrm{K13}$)] $[z,a^3 za^{-3}]=1$
\end{itemize}

\begin{observation}
 ($\textrm{K10}$) through ($\textrm{K13}$) hold in $K$.
\end{observation}

\begin{proof}
 Convert everything back into the generators $x_s$, $y_s$. We have $z = y_{110}y_{10}^{-1}$, $ada^{-1} = x_{00}$, $a^2 da^{-2} = x_{000}$, $a^2 za^{-2} = y_{01}y_{001}^{-1}$ and $a^3 za^{-3} = y_{001}y_{0001}^{-1}$. From this it is clear that ($\textrm{K10}$) through ($\textrm{K13}$) are true statements.
\end{proof}

\begin{proposition}[Finite set of relations]
 The relations ($\textrm{K1}_n$) through ($\textrm{K9}_n$) ($n\in\Z$) are all deducible from the eleven relations ($\textrm{K1}_0$), ($\textrm{K2}_0$), ($\textrm{K4}$), ($\textrm{K5}_0$), ($\textrm{K6}_0$), ($\textrm{K8}$), ($\textrm{K9}_0$), ($\textrm{K10}$), ($\textrm{K11}$), ($\textrm{K12}$) and ($\textrm{K13}$).
\end{proposition}

\begin{proof}
 First look at ($\textrm{K1}_n$). Applying ($\textrm{K4}$), we find that ($\textrm{K1}_n$) is equivalent to
 $$d^{-1}a^{-1}d^{-1}az^n ada^{-1}z^{-n} a^{-1}da = 1$$
 for all $n\in\Z$. Now applying ($\textrm{K10}$), this is equivalent to
 $$d^{-1}a^{-1}d^{-1}a^2 da^{-2}da = 1 \text{.}$$
 Since this holds for all $n\in\Z$, and the last expression is independent of $n$, we conclude that the family ($\textrm{K1}_n$) is derivable from just ($\textrm{K1}_0$), ($\textrm{K4}$) and ($\textrm{K10}$).
 
 Now look at ($\textrm{K2}_n$). Applying ($\textrm{K4}$) and conjugating, ($\textrm{K2}_n$) is equivalent to
 $$z^n ad^{-1}a^{-1}z^{-n}a^{-1}d^{-1}az^n az^n ada^{-1}z^{-n}a^{-1}z^{-n} a^{-1}da = 1$$
 for all $n\in\Z$. Now we apply ($\textrm{K10}$) to get
 $$ad^{-1}a^{-2}d^{-1}az^n a^2 da^{-2}z^{-n} a^{-1}da = 1$$
 and then ($\textrm{K11}$) to get
 $$ad^{-1}a^{-2}d^{-1}a^3 da^{-3}da = 1 \text{.}$$
 This is independent of $n$, so we conclude the family ($\textrm{K2}_n$) is derivable from ($\textrm{K2}_0$), ($\textrm{K4}$), ($\textrm{K10}$) and ($\textrm{K11}$).
 
 The next family is ($\textrm{K5}_n$). After conjugating and applying ($\textrm{K8}$), we see ($\textrm{K5}_n$) is equivalent to
 $$az^{-1}a^{-1}z^{-1}a^{-1}d^{-1}a^2 za^{-1}z^{n+1}a^2 za^{-2}z^{-n}a^{-1}da = 1$$
 for any $n\in\Z$. Now we apply ($\textrm{K12}$) and get
 $$az^{-1}a^{-1}z^{-1}a^{-1}d^{-1}a^2 za^{-1}za^2 za^{-3}da = 1 \text{,}$$
 which is derivable from ($\textrm{K5}_0$), ($\textrm{K8}$) and ($\textrm{K12}$).
 
 Now we move to ($\textrm{K6}_n$). After some cyclic reduction, ($\textrm{K6}_n$) is equivalent to
 $$z^n az^n az^{-1}a^{-1}z^{-(n+1)}a^{-1}z^{-(n+1)}a^{-1}d^{-1}az^{n+1}az^{n+1}az^{n+1}aza^{-1}z^{-n}a^{-1}z^{-n}a^{-1}z^{-n}a^{-1}da = 1 \text{.}$$
 Repeated applications of ($\textrm{K8}$), ($\textrm{K12}$) and ($\textrm{K13}$) yield
 $$z^{-1}a^2 z^{-1}a^{-1}z^{-1}a^{-2}d^{-1}azazazaza^{-4}da = 1 \text{,}$$
 which is derivable from ($\textrm{K6}_0$), ($\textrm{K8}$), ($\textrm{K12}$) and ($\textrm{K13}$).
 
 The last family is ($\textrm{K9}_n$). We start with
 $$a^{-1}d^{-1}az^n d^{-1}z^{-n}a^{-1}daz^n az^{-1}a^{-1}z^{-(n+1)}a^{-1}d^{-1}az^{n+1}aza^{-1}z^{-(n+1)}a^{-1}d^2az = 1$$
 and apply ($\textrm{K4}$) and ($\textrm{K8}$) to get
 $$a^{-1}d^{-1}ad^{-1}a^{-1}da^2 z^{-1}a^{-1}z^{-1}a^{-1}d^{-1}a^2 za^{-2}d^2az = 1 \text{,}$$
 which is derivable from ($\textrm{K4}$), ($\textrm{K8}$) and ($\textrm{K9}_0$).
\end{proof}

This proves that $K$ is finitely presented, which was conjecture ($\textrm{C}_2$), and so Theorem~\ref{thrm:BNSR} holds for $m=2$.

\bibliographystyle{alpha}

\end{document}